\documentclass{amsart}
\pdfoutput=1
\usepackage[utf8]{inputenc} 

\usepackage[T1]{fontenc}    

\usepackage{url}

\usepackage{amsmath}
\usepackage{amsfonts}
\usepackage{amssymb}
\usepackage{amsthm} %
\usepackage{mathrsfs} %
\usepackage{enumerate}

\usepackage{tikz}
\usepackage{pgfplots}
\usetikzlibrary{calc}
\usetikzlibrary{shapes}

\usetikzlibrary{intersections}
\usetikzlibrary{patterns}

\usepackage{graphicx}
\usepackage{caption}
\usepackage[hyperfootnotes,colorlinks=true,citecolor=cyan,backref=page]{hyperref}

\theoremstyle{plain} 
\newtheorem{thm}{Theorem}[section]
\newtheorem{cor}[thm]{Corollary}
\newtheorem{prop}[thm]{Proposition}
\newtheorem{lem}[thm]{Lemma}
\newtheorem{conj}{Conjecture}

\theoremstyle{definition}
\newtheorem{defi}[thm]{Definition}

\newtheorem{remark}[thm]{Remark}
\newtheorem{remarks}[thm]{Remarks}

\newtheorem{ex}[thm]{Example}

\newcommand{\coatom}{\operatorname{coatom}}
\newcommand{\atom}{\operatorname{atom}}

\newcommand{\rank}{\operatorname{rk}}
\newcommand{\cone}{\operatorname{cone}}

  \newcommand{\Sym}[1]{\mathcal{S}_{#1}}
\newcommand{\Symn}{\Sym{n}}
\newcommand{\NN}{\mathbb{N}}
\DeclareMathOperator{\std}{std}
\newcommand{\rest}[2]{#1_{|#2}}
\newcommand{\srest}[2]{\std(#1_{|#2})}
\DeclareMathOperator{\invs}{Inv}
\newcommand{\dec}[1]{\widetilde{#1}}

\def\R{{\overrightarrow{n}}}
\def\L{{\overleftarrow{n}}}
\def\d{d_R}

\newcommand{\Bsetnn}{[\![\bar{n},n ]\!]^*}
\newcommand{\Bsetkk}{[\![\bar{k},k ]\!]^*}
\newcommand{\Bstd}{\std^{B}}
\newcommand{\Bsrest}[2]{\Bstd(\rest{#1}{#2})}
\def\F{\mathcal{F}}
\newcommand{\red}[1]{\textcolor{red}{#1}}
\newcommand{\blue}[1]{\textcolor{blue}{#1}}

\makeatletter
\renewcommand\paragraph{\@startsection{paragraph}{4}{\z@}{2ex \@plus.5ex \@minus.2ex}{-1em}{\normalfont\normalsize\bfseries}}
\makeatother

\author[C. Hohlweg]{Christophe~Hohlweg}
\address[Christophe Hohlweg]{Universit\'e du Qu\'ebec \`a Montr\'eal\\
LaCIM et D\'epartement de Math\'ematiques\\ CP 8888 Succ. Centre-Ville\\
Montr\'eal, Qu\'ebec, H3C 3P8\\ Canada}
\email{hohlweg.christophe@uqam.ca}
\urladdr{http://hohlweg.math.uqam.ca}

\author[V. Pons]{Viviane~Pons}
\address[Viviane Pons]{Universit\'e Paris-Saclay, CNRS, Laboratoire Interdisciplinaire des Sciences du Num\'erique, Orsay, France.}
\email{viviane.pons@lisn.upsaclay.fr}

\title[A conjecture on descents, inversions and the weak order]{A conjecture on descents, inversions and the weak order on Coxeter groups}

\keywords{Coxeter groups, inversion sets, weak order}
\thanks{This work was partially supported by the NSERC  grant {\em combinatorics of infinite Coxeter groups} held by Hohlweg.}
\subjclass[2020]{Primary 20F55; secondary 05A05, 17B22; 05E16}

\begin{document}

\begin{abstract} In this article, we discuss the notion of partition of elements in an arbitrary Coxeter system $(W,S)$: a partition of an element $w$ is a subset $\mathcal P\subseteq W$ such that the left inversion set of $w$ is the disjoint union of the left inversion set of the elements in $\mathcal P$. Partitions of elements of $W$ arises in the study of the Belkale-Kumar product on the cohomology $H^*(X,\mathbb Z)$, where $X$ is  the complete flag variety of any complex semi-simple algebraic group. Partitions of elements in the symmetric group $\mathcal S_n$ are also related to the {\em Babington-Smith model} in algebraic statistics or to the simplicial faces of the Littlewood-Richardson cone. 

Moreover, we state and discuss the conjecture that the number of right descents of $w$ is the  sum of the number of right descents of the elements of~$\mathcal P$. In particular: (1) we state an equivalent conjecture in term of the number of atoms  in the interval $[e,w]_R$ of the right weak order, with a word-metric flavour; (2) we give a direct proof that this conjecture holds in the cases of symmetric groups (type $A$) and hyperoctahedral groups (type $B$). 
\end{abstract}
\date{\today}
\maketitle

\setcounter{tocdepth}{1}
\tableofcontents

\section{Introduction}   Let $(W,S)$ be a Coxeter system with length function $\ell:W\to \mathbb N$. The {\em set of reflections} of $W$ is the set:
$
T=\bigcup_{w\in W}wSw^{-1}.
$
 The {\em inversion set} (as a subset of reflections) of $w\in W$ is:
 $$
 T(w) = \{t \in T\mid \ell(tw)<\ell(t)\}.
 $$

\begin{defi}[Partitions of elements of $W$]\label{def:1} Let $w\in W$. A {\em partition of $w$} is a subset $\mathcal P\subseteq W$ such that:
$$
T(w)=\bigsqcup_{u\in \mathcal P} T(u) \quad\textrm{(disjoint union)}.
$$
We use the following terminology:
\begin{itemize}
\item A {\em $k$-partition of $w$} is a partition of $w$ of cardinality $k\in\mathbb N^*$. 
\item A partition $\mathcal P$ of $w\in W$ is called {\em a proper partition of $w$} if $e\notin \mathcal P$; $\{e,w\}$ is a $2$-partition of $w$ that is not proper.
\item A proper $2$-partition of $w$ is called a {\em bipartition of $w$}. 
\item The element $w$ is {\em partition-irreducible} if it does not admit any proper $k$-partition with $k>1$. By convention the identity $e$ is partition-irreducible. 
\end{itemize}
\end{defi}

Partitions of elements of $W$ are called {\em decompositions of inversion sets} in \cite{DeDiRo17} and {\em inv-decomposition} in \cite{Ka13}. Partitions of elements of $W$ arises in the study of the Belkale-Kumar product~\cite{BeKu06} on the cohomology $H^*(X,\mathbb Z)$, where $X$ is  the complete flag variety of any complex semi-simple algebraic group; see also \cite{DiRo09,DiRo17}. Partitions of elements in the symmetric group $\mathcal S_n$ are also related to the {\em Babington-Smith model} in algebraic statistics, see \cite{Ka13} for more details, and to the simplicial faces of the Littlewood-Richardson cone~\cite{DeDiRo17}.

\smallskip
 The {\em right descent set of $w\in W$} is $D_R(w)=\{s\in S \mid \ell(ws)<\ell(w)\}$; the number of right descent of $w$ is  denoted by $d_R(w)$  The aim of this article is to study the following conjecture. 

\begin{conj}\label{conj:1}  Let  $(W,S)$ be a Coxeter system. If  $\{u,v\}$ is a bipartition of $w\in W$, then $d_R(w)=d_R(u)+d_R(v)$. 
\end{conj}

Conjecture~\ref{conj:1} is a generalization of a conjecture formulated by Ressayre in the case of finite Weyl groups, which the first author was made aware of in~2014; see \S\ref{ss:Weyl} for more details on Ressayre's original conjecture.

The validity of Conjecture~\ref{conj:1} guarantees such an equality of sum of right descent sets also for arbitrary partitions of an element of $W$; the following proposition  is proven in \S\ref{se:1}.

\begin{prop}\label{prop:Conj} Let $(W,S)$ be a Coxeter system for which Conjecture~\ref{conj:1} holds. Let $\{u_1,\dots,u_k\}$ be a $k$-partition of $w$. Then $d_R(w)=\sum_{i=1}^k d_R(u_i)$.
\end{prop}

As far as we know, Conjecture~\ref{conj:1} holds for dihedral groups (Example~\ref{ex:1} below), for universal Coxeter systems, since all elements in that case are partition-irreducible  (Proposition~\ref{prop:Conj1True}), as well as for finite Weyl group, as a consequence of~\cite{FrRe23}. However, in the case of finite Weyl group, no direct\footnote{A proof involving solely argument from the theory of Coxeter groups.} proof of Conjecture~\ref{conj:1} is known, see the discussion in \S\ref{ss:Weyl}.  

 In this article we provide a direct proof of Conjecture~\ref{conj:1} for type $A$ in \S\ref{se:Sn} and for type $B$ in \S\ref{se:Bn}, which we believe have a some potential to be generalized to arbitrary Coxeter systems (see Remark~\ref{rem:GeneralizeAB}).  We have also made numerous computations with Sagemath~\cite{sage} to check the validity of Conjecture~\ref{conj:1} in infinite Coxeter groups, see \S\ref{ss:Computations}. 

\begin{ex}\label{ex:1}  Let $W$ be a dihedral group $\mathcal D_m$ ($m\in\mathbb N_{\geq 2}\cup\{\infty\})$ generated by $S=\{s,t\}$ with Coxeter graph:
\begin{center}
\begin{tikzpicture}[sommet/.style={inner sep=2pt,circle,draw=blue!75!black,fill=blue!40,thick}]
	\node[sommet,label=above:$s$] (alpha) at (0,0) {};
	\node[sommet,label=above:$t$] (beta) at (1,0) {} edge[thick] node[auto,swap] {$m$} (alpha);
\end{tikzpicture}
\end{center}
If $m\not=\infty$, only the longest element $w_\circ = sts\cdots=tst\cdots$ ($m=\ell(w_\circ)$ letters)  admits proper bipartitions. Indeed, if $w\in W$ admits a proper partition $\mathcal P$ with $|\mathcal P|>1$, then each for $u\in\mathcal P$, $T(u)$ contains at least an element of $S$ (a left descent), since $e\not\in \mathcal P$. Because the union of the inversion sets of elements in a partition is disjoint,   $|\mathcal P|=2$ and $s,t\in T(u)\sqcup T(v)=T(w)$. So $w=w_\circ$.  All the proper bipartitions of $w_\circ$ are therefore of the form $\{u=st\cdots, v=ts\cdots\}$ with $\ell(u)=i$, $\ell(v)=m-i$ and $1\leq i<m$.   Therefore we have $d_R(w_\circ)=2=1+1=d_R(u)+d_R(v)$.
In particular, the set of partition-irreducible elements is $W\setminus\{w_\circ\}$ in this case.

The same argument as above shows that, if $m=\infty$, then all elements of $W$ are partition-irreducible. 
\end{ex}

We show in Proposition~\ref{prop:Longest} that the longest element $w_\circ$ of a finite Coxeter system is never partition-irreducible and that it satisfies Conjecture~\ref{conj:1}.

\begin{ex}\label{ex:2} Consider $W$ to be the symmetric group $\mathcal S_4$  generated by the simple transpositions $S=\{\tau_1,\tau_2,\tau_3\}$. Therefore $(\mathcal S_4,S)$ is a Coxeter graph of type $A_{n-1}$, i.e., its  Coxeter graph is:
\begin{center}
\begin{tikzpicture}[sommet/.style={inner sep=2pt,circle,draw=blue!75!black,fill=blue!40,thick}]
	\node[sommet,label=above:$\tau_1$] (alpha) at (0,0) {};
	\node[sommet,label=above:$\tau_2$] (beta) at (1,0) {} edge[thick] node[auto,swap] {} (alpha);
	\node[sommet,label=above:$\tau_3$] (gamma) at (2,0) {} edge[thick] node[auto,swap] {} (beta);
\end{tikzpicture}
\end{center}
The set of reflections is $T=S\sqcup\{\tau_{(13)}=\tau_1\tau_2\tau_1,\tau_{(24)}=\tau_2\tau_3\tau_2,\tau_{(14)}=\tau_1\tau_2\tau_3\tau_2\tau_1\}$. 

Consider $w=\tau_{(14)}=\tau_3\tau_2\tau_1\tau_2\tau_3$. Then $w$ has two proper bipartitions:  
$$
\{u=\tau_1\tau_2,v=\tau_3\tau_2\tau_1\}\textrm{ and }\{\tau_3\tau_2,\tau_1\tau_2\tau_3\}.
$$ 
Indeed, for $\{u,v\}$ for instance we have:
$$
T(w)=\{\tau_1, \tau_{(13)}\}\sqcup\{\tau_{(14)},\tau_{(24)},\tau_{3}\}=T(\tau_1\tau_2)\sqcup T(\tau_3\tau_2\tau_1)=T(u)\sqcup T(v).
$$
Then $D_R(w)=\{\tau_1,\tau_3\}$, $D_R(u)=\{\tau_2\}$, $D_R(v)=\{\tau_1\}$ and $d_R(w)=d_R(u)+d_R(v)$.

Consider now $w'=\tau_1\tau_3\tau_2\tau_1=\tau_3\tau_2\tau_1\tau_2$. Then $\{\tau_1,v\}$ is a bipartition of $w'$ since:
$$
T(w')=\{\tau_1\}\sqcup\{\tau_{(14)},\tau_{(24)},\tau_{3}\}=T(\tau_1)\sqcup T(v).
$$
Again in this case we have $d_R(w)=d_R(u)+d_R(v)$. Notice that $D_R(w')=\{\tau_1,\tau_2\}$ and $D_R(\tau_1)=\{\tau_1\}=D_R(v)$.
\end{ex}

\begin{remark} If $\{u,v\}$ is a bipartition of $w$,  $D_R(w)\not = D_R(u)\sqcup D_R(v)$ in general, as observed in Example~\ref{ex:2} above. 
\end{remark}

\subsection{Conjecture~\ref{conj:1} and the right weak order}\label{ss:Weak}   The computations to check the validity of Conjecture~\ref{conj:1} with Sagemath~\cite{sage}, for which we present some enumerative results in~\S\ref{ss:Computations},  were made using an equivalent description for bipartitions of elements  in intervals in the (right) weak order, endowed with the word-metric. Recall that the {\em (right) weak order} is the poset $(W,\leq_R)$  defined by $u\leq_R w$ if and only if $u$ is a {\em prefix} of $w$, i.e., a reduced word for $u$ appear as the prefix of a reduced word for $w$. This is equivalent to consider the right Cayley graph of $(W,S)$ with the edges oriented as followed: an edge labeled by $s\in S$ is oriented from $w$ to $ws$ if and only if $s\notin D_R(w)$, if and only if $\ell(ws)>\ell(w)$. 

It is well-known that the Cayley graph of $(W,S)$ is a metric space for the {\em word metric} $d:W\times W\to \mathbb N$ defined by  $d(u,v)=\ell(u^{-1}v)$. 

\begin{defi}\label{def:2} Let $w\in W$. A subset $\{u,v\}$ of the interval  $[e,w]_R$ in $(W,\leq_R)$ is a {\em diameter of  $[e,w]_R$} if $d(u,v)=\ell(w)$. In particular, the pair $\{e,w\}$ is a diameter. The interval $[e,w]_R$ is called {\em rectangled} if it has at least two diameters. 
\end{defi}

An example of a rectangled interval is given in Figure~\ref{fig:D4}. It turns out that $\ell(w)$ is the maximum of all $d(u,v)$ for $u,v\in [e,w]_R$ (Proposition ~\ref{cor:MaxInt}),  hence the choice of the term {\em diameter}. We show in~\S\ref{ss:Rectangled} the following proposition, which provide a first answer to  \cite[Question~6]{DiRo09}.

\begin{prop}\label{prop:rectangled} Let  $u,v,w\in W$, then $\{u,v\}$ is a bipartition of $w$ if and only if $\{u,v\}$ is a diameter of $[e,w]_R$. In particular, the interval $[e,w]_R$ is rectangled if and only if $w$ is not partition-irreducible.
\end{prop}

The above proposition allows for two useful reformulations of Conjecture~\ref{conj:1}.  For a non-empty interval $[u,v]_R$ in the right weak order $(W,\leq_R)$, we denote by:
\begin{itemize}
\item $\atom([u,v]_R)$ the number of  {\em atoms} in $[u,v]_R$, i.e., the number of elements in $[u,v]_R$ that cover  $u$; 
\item $\coatom([u,v]_R)$ the number of  {\em coatoms} in $[u,v]_R$, i.e., the number of elements in $[u,v]_R$ that are covered by $v$.
\end{itemize}

\begin{conj}\label{conj:2} Let $u,v,w\in W$ such that $\{u,v\}$ is a diameter of $[e,w]_R$. Then $\coatom([e,w]_R)= \coatom([e,u]_R) + \coatom([e,v]_R)$.
\end{conj} 

\begin{conj}\label{conj:3} Let $u,v,w\in W$ such that $\{u,v\}$ is a diameter of $[e,w]_R$. Then $\atom([e,w]_R)= \atom([u,w]_R) + \atom([v,w]_R)$.
\end{conj} 

Since $d_R(g)=\coatom([e,g]_R)$ for any $g\in W$, Conjecture~\ref{conj:1} is equivalent to Conjecture~\ref{conj:2}, by Proposition~\ref{prop:rectangled}. However, Conjecture~\ref{conj:3} is a switch of perspective: it states if true that the number of {\em left descents} of $w$ is equal to the number of {\em right ascents} of $u$ plus the number of {\em right ascents} of $v$ such that the geodesics they start remains in the interval $[e,w]_R$. The following theorem is proven in \S\ref{ss:Rectangled}.

\begin{thm}\label{thm:Equivalent} Conjecture~\ref{conj:1}, Conjecture~\ref{conj:2} and Conjecture~\ref{conj:3} are equivalent. 
\end{thm}

Since Conjecture~\ref{conj:1} holds for finite Weyl groups, in particular for symmetric and hyperoctahedral groups,  we obtain the following corollary, for which it would be interesting to have a proof involving only weak order technics.

\begin{cor} Assume that $W$ is a finite Weyl group. Then for any~$w\in  W$ and $\{u,v\}$ a  diameter of $[e,w]_R$, we have:
\begin{enumerate}
\item $\atom([e,w]_R)= \atom([u,w]_R) + \atom([v,w]_R)$;
\item $\coatom([e,w]_R)= \coatom([e,u]_R) + \coatom([e,v]_R)$.
\end{enumerate}
\end{cor}

\begin{figure}[h!]
\resizebox{0.8\hsize}{!}{
\begin{tikzpicture}
	[scale=2,
	 q/.style={teal,line join=round},
	 racine/.style={blue},
	 racinesimple/.style={blue},
	 racinedih/.style={blue},
	 sommet/.style={inner sep=2pt,circle,draw=black,fill=blue,thick,anchor=base},
	 rotate=0]
 \tikzstyle{every node}=[font=\small]
\def\grosseursimple{0.025}
\node[anchor=south west,inner sep=0pt] at (0,0) {\includegraphics[width=6.4cm]{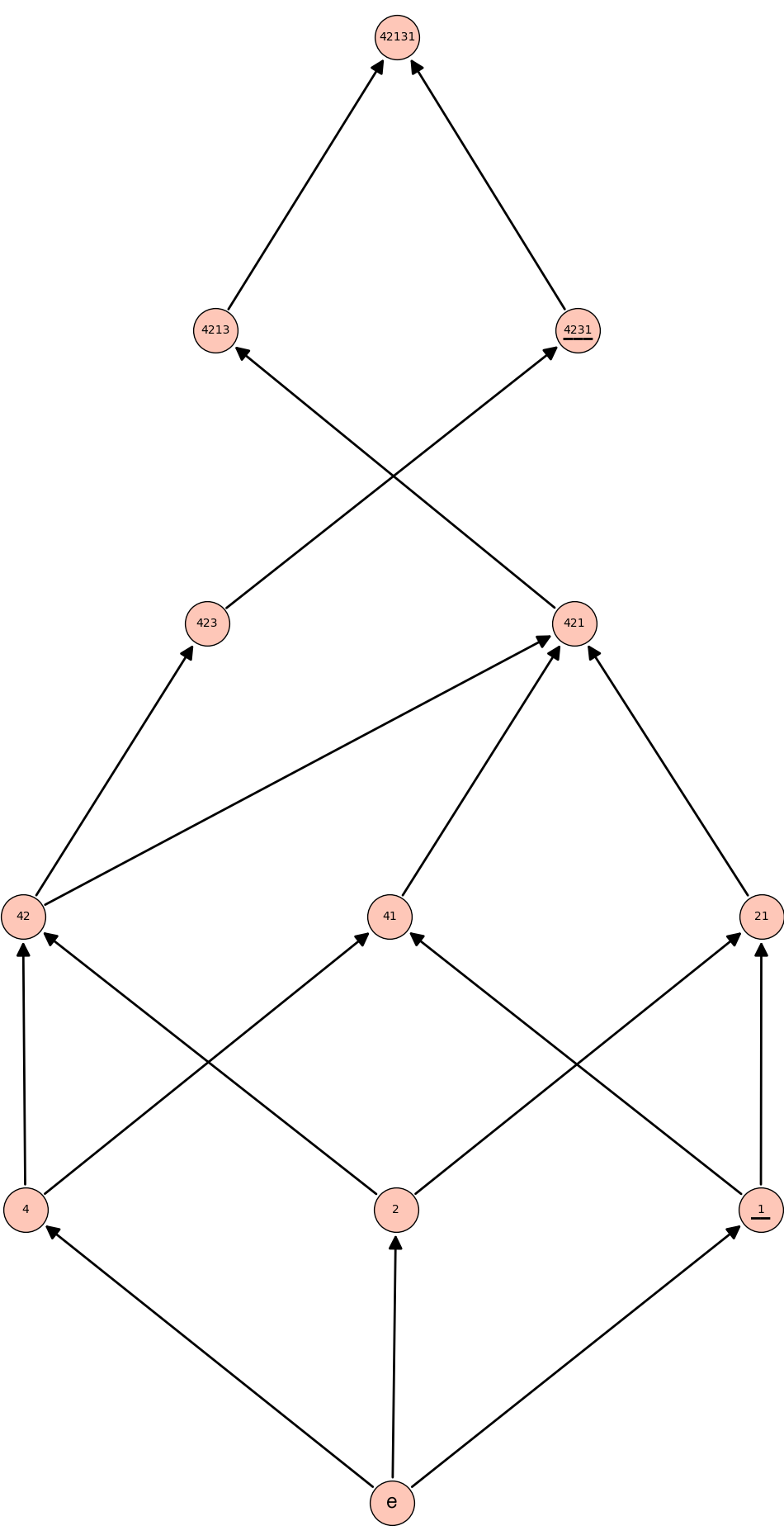}};
\coordinate (ancre) at (-0.6,5);

\node[sommet,label=above:$1$] (gamma) at ($(ancre)+(0.5,0.43)$) {};
\node[sommet,label=below :$3$] (beta) at ($(ancre)+(0.5,0)$) {} edge[thick] node[auto,swap] {} (gamma) ;
\node[sommet,label=below :$4$] (muph) at ($(ancre)+(1,0)$) {} edge[thick] node[auto,swap] {} (beta) ;
\node[sommet,label=below:$2$] (alpha) at (ancre) {}  {} edge[thick] node[auto,swap] {} (beta) ;


\end{tikzpicture}}
\caption{The rectangled interval $[e,42131]_R$ for $(W,S)$ of type $D_4$ with Coxeter graph on the upper left side. This rectangled interval has only two diameters: $\{e,42131\}$ and  $\{1,4231\}$.  The coatoms of $[e,42131]_R$ are $4213,4231$, the coatom of $[e,1]_R$ is $e$ and the coatoms of $[e,4231]_R$ is $423$, so Conjecture~\ref{conj:2}, and therefore Conjecture~\ref{conj:1}, hold in this example. The atoms of $[e,42131]_R$ are $1,2,4$ (the left descents), the atoms of $[1,42131]_R$ are $41,21$ and the atom of $[4231,42131]_R$ is $42131$, which illustrates the statement of Conjecture~\ref{conj:3}.}
\label{fig:D4}
\end{figure}

\subsection{Conjecture~\ref{conj:1} is true for finite Weyl groups}\label{ss:Weyl} As stated above, Conjecture~\ref{conj:1} is a generalization of a conjecture formulated by Ressayre in the case of finite Weyl groups, which the first author was made aware of in~2014. Let $W$ be a finite Weyl group with longest element $w_\circ$ and let $u_1,u_2,u_3\in W$, Ressayre conjectured the following statement:
$$
T(w_\circ)=T(u_1)\sqcup T(u_2)\sqcup T(u_3)\implies \rank(W)=|S|=d_R(u_1)+d_R(u_2)+d_R(u_3)\qquad (\triangle).
$$
In other words, if $\{u_1,u_2,u_3\}$ is a $3$-partition of $w_\circ$ then $|S|=d_R(u_1)+d_R(u_2)+d_R(u_3)$.  

In 2023, Francone and Ressayre~\cite[Theorem~1]{FrRe23}  showed that the structure coefficients of the Belkale-Kumar product are either $0$ or $1$ for any finite Weyl group, answering by the positive~\cite[Claim~1]{DiRo09}. Their theorem implies  $(\triangle)$ as a corollary~\cite[Corollary~5]{FrRe23}. Their proof relies on an algebraic geometry argument, based on the fact that complete flag varieties are simply connected, to reduce the problem to a case-by-case computation on crystallographic root systems spanning over twenty pages. It turns out that Conjecture~\ref{conj:1} is equivalent to statement $(\triangle)$, see Proposition~\ref{prop:Ressayre}. So Conjecture~\ref{conj:1} is true for finite Weyl groups as a consequence of~\cite[Theorem~1]{FrRe23}. 

\smallskip
Unfortunately, this equivalent statement $(\Delta)$ is not valid for infinite Coxeter systems (since the longest element $w_\circ$ is involved), whereas Conjecture~\ref{conj:1} seems to be. So direct or undirected proofs of $(\Delta)$ won’t likely be generalizable to arbitrary Coxeter system. But still a direct proof of statement $(\Delta)$ for finite Coxeter groups could be interesting to better understand  Conjecture~\ref{conj:1}.

In type $A$, a proof\footnote{The authors of \cite{DeDiRo17} states that their result can be extended to type $B$.}  of $(\Delta)$, involving solely the structure of the corresponding crystallographic root system, is given thanks to \cite[Proposition 8.1]{DeDiRo17}. The idea is as follows. For $w\in W$, denote by $\Phi^R(w)=\{-w(\alpha_s)\mid s\in D_R(w)\}$, the so called set of {\em descent-roots}. In \cite[Proposition 8.1]{DeDiRo17}, the authors show that if $\{u_1,u_2,u_3\}$ are a $3$-partition of $w_\circ$, then $\Phi^R(u_1)\sqcup\Phi^R(u_2)\sqcup\Phi^R(u_3)$ is  a basis of the vector space spanned by $\Phi$, which is of dimension $|S|$. Hence 
$$
|S|=|\Phi^R(u_1)|+|\Phi^R(u_2)|+|\Phi^R(u_3)|=d_R(u_1)+d_R(u_2)+d_R(u_3).
$$

Unfortunately, we do not see how to adapt this argument in the case of Conjecture~\ref{conj:1}. Indeed, in Example~\ref{ex:2} (type $A_3$),  the element $w=\tau_1\tau_3\tau_2\tau_1\tau_3$ has  $u=\tau_1\tau_2$ and $v= \tau_3\tau_2\tau_1$ as a bipartition. Consider the crystallographic root system of $A_3$ with positive roots $e_j-e_i$ with $1\leq i<j\leq 4$. 
The unique  inversion-descent of $u$ is $e_3-e_1$, the unique one for $v$ is $e_4-e_1$ - and those form a base of a subspace $V_1$. However, the inversion descents of $w$ are $e_3-e_1$ and $e_4 - e_2$, which form a basis of a subspace $V_2$, distincts of  $V_1$.  If $V_1$ would have been equal to $V_2$, we could have concluded as in  \cite[Proposition 8.1]{DeDiRo17}. This phenomenon is not an exception and happens in most cases of bipartitions.
	It would be interesting to know if a statement similar to \cite[Proposition 8.1]{DeDiRo17} holds for bipartitions of elements of $w$. The proof of \cite[Proposition 8.1]{DeDiRo17} is not elementary and is a consequence of a long and technical analysis of some properties of crystallographic root systems of type~$A$.

\subsection*{Acknowledgements} Part of this research was done at the Laboratoire d'alg\`ebre, de combinatoire et d'informatique math\'ematique (LACIM) at UQAM during the sabbatical of the second author. Therefore, she warmly thanks the IRL CRM-CNRS of Montr\'eal and  the INSMI for giving her this opportunity. The first author warmly thanks Riccardo Biagioli, Matthew Dyer, Philippe Nadeau, Piotr Przytycki, Nicolas Ressayre and Paolo Sentinelli for instructive discussions.

\section{Partitions of elements in a Coxeter system}\label{se:1}

For a general reference on Coxeter groups, we refer the reader to~\cite{BjBr05}. For $u,v,w\in W$, we say that a product $w=uv$ is {\em a reduced product} if $\ell(w)=\ell(u)+\ell(v)$. In this case, we say that $u$ is {\em a prefix of $w$} and $v$ is a {\em suffix of $w$}.  

In  the weak order $(W,\leq_R)$, it is well-known that, for $u,w\in W$, we have $u\leq_R w$ if and only if $T(u)\subseteq T(w)$. Bj\" orner showed that $(W,\leq_R)$ is a complete meet-semilattice.  For $X\subseteq W$, we denote the {\em meet of $X$} by $\bigwedge_R X$   and, if it exists, the {\em join of $X$} by $\bigvee_R X$.

\subsection{First properties of partitions of elements of $W$} Companion to the right descent set is the {\em left descent set of $w\in W$}:
$$
D_L(w)=\{s\in S\mid \ell(sw)<\ell(w)\}=\{s\in S\mid s\leq_R w\}\subseteq T(w).
$$
We denote by $d_L(w)=|D_L(w)|$. The following proposition states, in particular, that the statement in  Conjecture~\ref{conj:1} obtained by replacing right descent sets  by left descents easily holds. 
 
\begin{prop}\label{prop:basic} Let $\mathcal P$ be a partition of $w\in W$, then:
\begin{enumerate}
\item $D_L(w)=\bigsqcup_{u\in \mathcal P} D_L(u)$. In particular $d_L(w)=\sum_{u\in\mathcal P} d_L(u)$;
\item $\ell(w)=\sum_{u\in \mathcal P} \ell(u)$;
\item $u\leq_R w$ for all $u\in\mathcal P$;
\item $\bigwedge_R \mathcal P = e$ and $\bigvee_R\mathcal P=w$.
\end{enumerate}
\end{prop}
\begin{proof}  (1) The statement is a direct consequences of the definition of partitions of $w$.  (2) The statement is a consequence of the definition of partitions of $w$ and of the fact that $|T(g)|=\ell(g)$ for any $g\in W$. (3) The statement is a consequence of the definition of partitions of $w$ and of the fact that $u\leq_R w$ if and only if $T(u)\subseteq T(w)$. (4) Assume by contradiction that $\wedge_R \mathcal P \not = e$. Then there is $s \in S$ such that $s\leq_R u$ for all $u\in\mathcal P$. So 
$$
s\in \bigcap_{u\in\mathcal P}D_L(u)\subseteq \bigcap_{u\in\mathcal P} T(u),
$$
contradicting that the union $T(u)$ for $u\in\mathcal P$ is disjoint. So $\bigwedge_R \mathcal P = e$. Finally, by (3), the set $\mathcal P$ is bounded by $w$ so $z:=\bigvee_R \mathcal P$ exists and $z\leq_R w$. So in particular, $T(z)\subseteq T(w)$. For all $u\in\mathcal P$, we have $u\leq_R z$ so $T(u)\subseteq T(z)$. Therefore: 
$$
T(w)= \bigsqcup_{u\in\mathcal P} T(u)\subseteq T(z)\subseteq T(w),
$$
implying $w=z$.
\end{proof}

As a direct consequence of Proposition~\ref{prop:basic}~(1) we obtain the following useful statement on partition-irreducible elements.

\begin{cor}\label{cor:Bi} Let $w\in W$ such that $d_L(w)=1$, then $w$ is partition-irreducible. In particular, if $w\in W$ admit a proper bipartition, then $d_L(w)>1$.
\end{cor}

Recall that $(W,S)$ is a {\em universal Coxeter system} if its Coxeter graph is a complete graph whose edges are all indexed by $\infty$. The following proposition states that universal Coxeter systems satisfy Conjecture~\ref{conj:1}, which is obviously true if there are only partition-irreducible elements in $W$.

\begin{prop}\label{prop:Conj1True} All elements in universal Coxeter systems are partition-irreducible. 
\end{prop}
\begin{proof} Suppose $(W,S)$ to be a universal Coxeter system. Let $w\in W$. It is well-known that if $D_L(w)=I$ then $W_I$ is finite. Since the only finite standard parabolic subgroups of a universal Coxeter system are the $W_{\{s\}}$ for $s\in S$, we conclude with Corollary~\ref{cor:Bi}. 
\end{proof}

\subsection{Roots and reflections} To discuss further properties of bipartition of elements of $W$ and Conjecture~\ref{conj:1}, we need the interpretation of inversion sets in a root system.

From now on  we consider a {\em geometric representation of $(W,S)$} over a quadratic space $(V,B)$. More precisely,  $V$ is a real vector space, $B$ is a symmetric bilinear form that is preserved by the action of $W$, and $S$ is mapped into a set of {\em $B$-reflections} associated to a {\em simple system} $\Delta=\{\alpha_s\mid s\in S\}$;  see for instance \cite[\S2.1]{HoLa16} for more details. Such a representation is a generalization of the {\em Tits canonical representation} as presented in \cite[\S4.1]{BjBr05}. 

For $X\subseteq V$, we denote by $\cone(X)$ the set of nonnegative linear combinations of vectors in $X$. We consider the associated root system $\Phi=W(\Delta)$, the positive root system $\Phi^+=\Phi\cap \cone(\Delta)$ and the negative root system $\Phi^-=-\Phi^+$. It is well-known that $\Phi=\Phi^+\sqcup \Phi^-$ and that the set of reflections $T$ is in bijection with the set of {\em $B$-reflections} associated to $\Phi^+$. In particular:
$$
T=\{s_\alpha\mid \alpha\in\Phi^+\}.
$$
The inversion set (as a subset of positive roots) of an element $w\in W$ is then:
$$
\Phi(w)=\{\alpha \mid s_\alpha \in T(w)\}=\Phi^+\cap w(\Phi^-).
$$
In particular $\ell(w)=|\Phi(w)|=|T(w)|$.

Subsets of $\Phi^+$ that are inversion sets are characterized by the notion of {\em biclosed sets}: a subset $A$ of $\Phi^+$ is {\em biclosed} if the two following conditions are met: (a) for any $\alpha,\beta \in A$, $\cone(\alpha,\beta)\subseteq A$; (b) Let $\alpha,\beta,\gamma\in \Phi^+$ such that $\gamma\in A\cap \cone(\alpha,\beta)$, then $\alpha\in A$ or $\beta\in A$.

We summarize in the next proposition some useful and well-known results about that inversion sets as a subset of positive roots; see for instance \cite[\S2.2-2.3]{HoLa16} for more details.

\begin{prop}\label{prop:Biclos}
\begin{enumerate}
\item A subset $A$ of $\Phi^+$ is biclosed if and only if $\Phi^+\setminus A$ is biclosed. 
\item Let $w=uv$ a reduced product, that is, $\ell(w)=\ell(u)+\ell(v)$, then $\Phi(w)=\Phi(u)\sqcup u(\Phi(v))$.
\item For $u,v\in W$, we have $\Phi(u)\subseteq \Phi(v)$ if and only if $u\leq_R v$. 
\item A finite subset $A$ of $\Phi^+$ is biclosed if and only if there is $w\in W$ such that $A=\Phi(w)$.
\end{enumerate}
\end{prop}

\subsection{Bipartitions of the longest element} Recall that $W$ is finite if and only if there is an element of maximal length, which is unique and denoted by $w_\circ$. It is well-known that the longest element has the following properties: $\Phi^+=\Phi(w_\circ)$; $w_\circ Sw_\circ = S$, $D_R(w_\circ)=D_L(w_\circ)=S$. Moreover,  for all $w\in W$, we have:
$$
\Phi(ww_\circ)=\Phi^+\setminus \Phi(w)\ \textrm{ and } \ell(ww_\circ)=\ell(w_\circ) - \ell(w).
$$ 

The following proposition states that Conjecture~\ref{conj:1} holds if $w=w_\circ$, the longest element of a finite Coxeter group $W$.

\begin{prop}\label{prop:Longest} Assume $W$ to be finite with $w_\circ$ as longest element.
\begin{enumerate} 
\item The set of bipartitions of $w_\circ$ is the set $\{\{u,uw_\circ\}\mid u\in W\}$.
\item Let  $\{u,v\}$ be a bipartition of  $w_\circ$, then 
$
|S|=d_R(w_\circ)=d_R(u)+d_R(v).
$
\end{enumerate}
\end{prop}
\begin{proof} We have $\Phi^+=\Phi(w_\circ)=\Phi(u)\sqcup \Phi(v)$ by definition of bipartitions. Write $u'=uw_\circ $, then by assumption: 
$$
\Phi(u')=\Phi^+\cap u'(\Phi^-)= \Phi^+\cap uw_\circ(\Phi^-)= \Phi^+\cap u(\Phi^+)=\Phi^+\setminus \Phi(u)=\Phi(v).
$$
So $v=u'=uw_\circ$, by Proposition~\ref{prop:Biclos}~(3). 

Now, let $s\in D_R(v)$. So $r=w_\circ s w_\circ \in S$ and we have:
$$
\ell(v)-1= \ell(vs)=\ell(uw_\circ s)=\ell(ur w_\circ)= \ell(w_\circ)-\ell(ur).
$$
Hence:
$$
\ell(ur)=\ell(w_\circ) - \ell(v)+1=\ell(w_\circ) - \ell(uw_\circ)+1=\ell(w_\circ) - (\ell(w_\circ)-\ell(u))+1=\ell(u)+1.
$$
So $r\notin D_R(u)$. The same line of reasoning shows that $s\in D_R(v)$ if and only if $w_\circ s w_\circ \in S\setminus D_R(u)$. Therefore, $d_R(v)=|S|-d_R(u)$, which concludes the proof since $D_R(w_\circ)=S$.
\end{proof}

The following result is key in the proof of Theorem~\ref{thm:Main} in the case of type $B$.

\begin{prop}\label{prop:ReductionLongest} Assume that $W$ is finite and let $u,v,w\in W$, then the following statements are equivalent:
\begin{enumerate}[(i)]
\item $\{u,v\}$ is a bipartition of $w$;
\item  $\{u,ww_\circ\}$ is a bipartition of $vw_\circ$;
\item $\{ww_\circ,v\}$ is a bipartition of $uw_\circ$.
\end{enumerate}
Moreover, if Conjecture~\ref{conj:1} holds in one of these cases, then it holds for all these cases. 
\end{prop}
\begin{proof} It is enough to show that $(i)$ implies $(ii)$. Suppose that $\{u,v\}$ is a bipartition of $w$. Then $\Phi(w)=\Phi(u)\sqcup \Phi(v)$. Therefore:
\begin{eqnarray*}
\Phi(vw_\circ)=\Phi^+\setminus\Phi(v)&=&(\Phi(w)\setminus \Phi(v)) \sqcup  ((\Phi^+\setminus\Phi(w))\setminus \Phi(v))\\
&=& \Phi(u) \sqcup  (\Phi^+\setminus\Phi(w))\\
&=& \Phi(u) \sqcup  \Phi(ww_\circ).
\end{eqnarray*}
Hence $\{u,ww_\circ\}$ is a bipartition of $vw_\circ$, proving the equivalence of the statements. 

Now, by Proposition~\ref{prop:Longest} we know that $d_r(gw_\circ)= |S|-d_R(g)$ for all $g\in W$. Therefore $d_R(w)=d_R(u)+d_R(v)$ if and only if $d_R(vw_\circ)=d_R(u)+d_R(ww_\circ)$, concluding the proof.
\end{proof}

The next statement shows, in particular, that in the case of finite Weyl groups Conjecture~\ref{conj:1} is equivalent to Ressayre's statement $(\triangle)$ from the introduction.

\begin{prop}\label{prop:Ressayre} Assume that $W$ is finite and let $u_1,u_2,u_3\in W$, then the following statements are equivalent:
\begin{enumerate}[(i)]
\item $\{u_1,u_2,u_3\}$ is a $3$-partition of $w_\circ$;
\item  $\{u_1,u_2\}$ is a bipartition of $w=u_3w_\circ$;
\end{enumerate}
Moreover, in this case, we have:
$$
d_R(u_1)+d_R(u_2)=d_R(w)\iff\rank(W)=|S|=d_R(u_1)+d_R(u_2)+d_R(u_3). 
$$
\end{prop}
\begin{proof} Write $w=u_3w_\circ$. If $\{u_1,u_2,u_3\}$ is a $3$-partition of $w_\circ$, then 
$$
\Phi(u_1)\sqcup\Phi(u_2)=\Phi(w_\circ)\setminus\Phi(u_3)=\Phi^+\setminus\Phi(u_3) = \Phi(w).
$$
The converse follows from the same computation. The last statement of the proposition follows from the fact that $d_R(w)=d_R(w_\circ)-d_R(u_3)$.
\end{proof}

\subsection{Proof of Proposition~\ref{prop:Conj}}  We need first to state the following lemma.

\begin{lem}\label{lem:PropConj} Let  $w\in W$ and $\{u_1,\dots,u_k\}$ be a $k-$partition of $w$. Then there is $v\in W$ such that $\Phi(v)=\Phi(u_2) \sqcup \cdots \sqcup \Phi(u_{k})$. In particular, $\{u_1,v\}$ is a bipartition of~$w$.
\end{lem}
\begin{proof} Set $A=\Phi(u_2) \sqcup \cdots \sqcup \Phi(u_{k})$.  We first show that $A$ is biclosed. 

\smallskip
\noindent (a) Let $\alpha,\beta \in A$. Assume by contradiction that there is $\gamma\in \cone(\alpha,\beta)\cap \Phi$ such that $\gamma\notin A$.  By Proposition~\ref{prop:Biclos}, $\Phi(u_1)$ is a biclosed set. Therefore, if $\gamma\in \Phi(u_1)$, then  $\alpha\in\Phi(u_1)$ or $\beta\in \Phi(u_1)$,  by (b) of the definition of biclosed sets; this is a contradiction since $A\cap \Phi(u_1)=\emptyset$. So $\gamma\notin  \Phi(u_1)\sqcup A=\Phi(w)$. Again, $\Phi(w)$ is biclosed, so by Proposition~\ref{prop:Biclos}, $\Phi^+\setminus \Phi(w)$ is biclosed. Therefore either $\alpha$ or $\beta$ must be in $\Phi^+\setminus \Phi(w)$, which is again a contradiction. We conclude that $\cone(\alpha,\beta)\subseteq A$.

\smallskip
\noindent (b)  Let $\alpha,\beta,\gamma\in \Phi^+$ such that $\gamma\in A\cap \cone(\alpha,\beta)$. Then $\gamma\in \Phi(u_i)$ for some $2\leq i\leq k$. Since $\Phi(u_i)$ is biclosed, we have $\alpha\in\Phi(u_i)\subseteq A$ or $\beta\in \Phi(u_i)\subseteq A$.

\smallskip
\noindent Finally, since $A$ is a finite biclosed set, we have by Proposition~\ref{prop:Biclos}~(4) that there is $v\in W$ such that $\Phi(v)=A=\Phi(u_2) \sqcup \cdots \sqcup \Phi(u_{k})$ and $\Phi(w)=\Phi(u_1)\sqcup \Phi(v)$. 
\end{proof}

\begin{proof}[Proof of Proposition~\ref{prop:Conj}] Assume that Conjecture~\ref{conj:1} holds for $(W,S)$. We prove the following statement by induction on $k\in\mathbb N^*$: Let  $w\in W$ and $u_1,\dots,u_k$ be a $k-$partition of $w$, then $d_R(w)=d_R(u_1)+\dots+d_R(u_k)$.

By definition of $k-$partitions, $\Phi(w)=\Phi(u_1)\sqcup \Phi(u_2) \sqcup \cdots \sqcup \Phi(u_{k})$.
The case $k=1$ is readily seen. The case $k=2$ is  Conjecture~\ref{conj:1}, which holds by assumption.

Assume $k\geq 3$. By Lemma~\ref{lem:PropConj}, there is $v\in W$ such that $\Phi(v)=\Phi(u_2) \sqcup \cdots \sqcup \Phi(u_{k})$ and $\Phi(w)=\Phi(u_1)\sqcup \Phi(v)$. We conclude by induction  that  $d_R(w)=d_R(u_1)+d_R(v)=d_R(u_1)+\dots+d_R(u_k)$. 
\end{proof}

\subsection{Conjecture~\ref{conj:1}, bipartitions and rectangled intervals in the weak order}\label{ss:Rectangled}

In this section, we discuss the characterization of bipartitions of elements as diagonals of intervals in the right weak order $(W,\leq_R)$, which leads to the equivalent statements between  Conjecture~\ref{conj:1} and Conjecture~\ref{conj:2}. 

Recall from the introduction that the distance between two elements of $u$ and $v$ of $W$ is $d(u,v)=\ell(u^{-1}v)$. 

\begin{lem}[{\cite[Lemma~1.2]{BrHo93}}]\label{lem:BH} Let $u,v\in W$, then $d(u,v)=\ell(u)+\ell(v)$ if and only if $\Phi(u)\cap\Phi(v)=\emptyset$. 
\end{lem}

The following result is due to Dyer~\cite{Dy90} and is a consequence of the fact that the map $w\mapsto T(w)$ satisfies the {\em cocycle condition}. This condition allows  to have a generalization of the formula in Proposition~\ref{prop:Biclos}~(2) for not necessarily reduced products. Let $A,B\subseteq T$ and denotes $A+B:=A\setminus B\sqcup B\setminus A$ the symmetric difference of $A$ and $B$. Dyer showed that the map $w\mapsto T(w)$ is characterized by the following two conditions:
\begin{enumerate}[(i)]
\item $T(s)=\{s\}$ for all $s\in S$;
\item  $T(uv)=T(u)+uT(v)u^{-1}$, for all $u,v\in W$.
\end{enumerate}
Moreover, note that  $T(w)=wT(w^{-1})w^{-1}$ for all $w\in W$. 

\begin{remark} This cocycle condition is naturally expressed with descent sets as sets of reflections (the $T(w)$'s) and not with descent sets as sets of roots (the $\Phi(w)$'s). The reason is that the conjugation action of $W$ on $T$ is equivalent to the natural action of $W$ on $\Phi$ as a subset of $V$; however this action on $\Phi$ does not restrict to an action on $\Phi^+$, since for instance $s(\alpha_s)=-\alpha_s$ for all $s\in S$. 
\end{remark}

\begin{lem}\label{lem:MaxInt} Let $u,v\in W$, then  $d(u,v)=|\Phi(u) + \Phi(v)|$. 
\end{lem}
\begin{proof} We have:
$$
d(u,v)=\ell(u^{-1}v)=|T(u^{-1}v)|=|T(u^{-1})+u^{-1}T(v)u|=|u^{-1}(uT(u^{-1})u^{-1}+T(v))u|.
$$
Therefore $d(u,v)=|uT(u^{-1})u^{-1}+T(v)|=|T(u)+T(v)|=|\Phi(u)+\Phi(v)|$, since the map $T(g)\mapsto \Phi(g)$ is a bijection preserving unions and intersections. 
\end{proof}

The following statement shows that the maximal distance in an interval $[e,w]_R$ in $(W,\leq_R)$ is achieved by a geodesic (a smallest path in the right Cayley graph) between $e$ and $w$.    

\begin{prop}\label{cor:MaxInt} Let $w\in W$  and $u,v\in [e,w]_R$, then $d(u,v)\leq \ell(w)$. 
\end{prop}
\begin{proof} Since $u\leq_R w$ and $v\leq_R w$, we have $\Phi(u)\cup\Phi(v)\subseteq \Phi(w)$, by Proposition~\ref{prop:Biclos}~(3). So 
$\ell(w)=|\Phi(w)|\geq |\Phi(u) + \Phi(v)|=d(u,v)$, by Lemma~\ref{lem:MaxInt}.
\end{proof}

\begin{remark} We believe that Proposition~\ref{cor:MaxInt} is well-known, but we could not find a reference so we included a proof. More generally, it is known that the interval $[e,w]_R$ is {\em convex}, i.e., any geodesic between two elements of $[e,w]_R$ is contained in  $[e,w]_R$. This follows from the interpretation of the inversion set $T(g)$, for $g\in W$,  as the set $\mathcal H(g)$ of hyperplanes separating $g$ from the identity $e$. One can show that, for  $u,v\in [e,w]_R$, the set of hyperplanes $\mathcal H(u,v)$ that separates $u$ from $v$ satisfies $d(u,v)=|\mathcal H(u,v)|$ and   $\mathcal H(u,v)\subseteq \mathcal H(u)\cup \mathcal H(v)\subseteq \mathcal H (w)$. Moreover, any $g\in W$ in a geodesic from $u$ to $v$ satisfies $\mathcal H(u,g) \sqcup \mathcal (g,v) =\mathcal H(u,v)\subseteq \mathcal H(w)$ and $\mathcal H(g)\setminus \mathcal H(u,v)= H(u)\cup \mathcal H(v)\setminus \mathcal H(u,v)$. Therefore $\mathcal H(g)\subseteq \mathcal H (w)$, implying $g\leq_R w$.
\end{remark}

\begin{proof}[Proof of Propostion~\ref{prop:rectangled}] Let $u,v,w\in W$. We need to show that $\{u,v\}$ is a bipartition of $w$ if and only if $\{u,v\}$ is a diameter of $[e,w]_R$. 

Assume first that $\{u,v\}$ is a bipartition of $w$. Therefore $\ell(w)=\ell(u)+\ell(v)$ by Proposition~\ref{prop:basic}~(2). Now, by Lemma~\ref{lem:BH}, $\Phi(u)\cap \Phi(v)=\emptyset$. Finally we obtain $d(u,v)=\ell(u)+\ell(v)=\ell(w)$ and, therefore, $\{u,v\}$ is a diameter of $[e,w]_R$. 

Assume now that $\{u,v\}$ is a diameter of $[e,w]_R$. On one hand, by definition and Lemma~\ref{lem:MaxInt}, we have:
$$
(\star)\qquad \ell(w)=d(u,v)=\ell(u^{-1}v)=|\Phi(u)+\Phi(v)|=|\Phi(u)|+|\Phi(v)|-2|\Phi(u)\cap\Phi(v)|.
$$  
On the other hand, $u,v\in [e,w]_R$ implies $\Phi(u)\cup\Phi(v)\subseteq \Phi(w)$, by Proposition~\ref{prop:Biclos}~(c). So
$$
\ell(w)\geq |\Phi(u)\cup \Phi(v)|=|\Phi(u)|+|\Phi(v)|-|\Phi(u)\cap \Phi(v)|.
$$
Together with $(\star)$, we obtain  $-|\Phi(u)\cap\Phi(v)|\geq 0$. Therefore $\Phi(u)\cap \Phi(v)=\emptyset$. So by $(\star)$ again, we have  $\ell(w) = |\Phi(w)|=|\Phi(u)\sqcup\Phi(v)|$. Since $\Phi(u)\cup\Phi(v)\subseteq \Phi(w)$, we conclude that $\Phi(w)=\Phi(u)\sqcup\Phi(v)$. In other words,  $\{u,v\}$ is a bipartition of $w$. 
 \end{proof}

We end this section with the following proposition that we need to prove the equivalence between Conjecture~\ref{conj:1},  Conjecture~\ref{conj:2} and Conjecture~\ref{conj:3}.

\begin{prop}\label{prop:IntWeak} Let $w\in W$. The map $\varphi_w: [e,w]_R\to [e,w^{-1}]_R$, defined by $\varphi_w(g)=w^{-1}g$, is a bijection; $\varphi_w{}^{-1} = \varphi_{w^{-1}}$. Moreover for any $u,v\in [e,w]_R$ we have: 
\begin{itemize}
\item $u\leq_R v$ if and only if $\varphi_w(u)\geq_R\varphi_w(v)$ ($\varphi_w$ is a poset anti-isomorphism);
\item  $d(\varphi_w(u),\varphi_w(v))=d(u,v).$ ($\varphi_w$ is an isometry).
\end{itemize}
In particular, $\{u,v\}$ is a diameter of $ [e,w]_R$ if and only if $\{(\varphi_w(u),\varphi_w(v)\}$ is a diameter of  $[e,w^{-1}]_R$.
\end{prop}
 \begin{proof} The map $\varphi$ is well-defined. Indeed, let $g\leq_R w$, so $w=gh$ is a reduced product (i.e.  $\ell(w)=\ell(g)+\ell(h)$). So $\varphi_w(g)= w^{-1}g=h^{-1}$. Since $w^{-1}=h^{-1}g^{-1}$ and  $\ell(w^{-1})=\ell(g^{-1})+\ell(h^{-1})$, $h^{-1}$ is a prefix of $w^{-1}$. In other words, $\varphi_w(g)=h^{-1}\leq_R w^{-1}$. It is readily seen that $\varphi_w{}^{-1} = \varphi_{w^{-1}}$ 
 
Let $u,v\in [e,w]_R$, then $w=ux$ and $w=vy$ are reduced products. We have $\varphi_w(u)=x^{-1}$ and $\varphi_w(v)=y^{-1}$. 

Assume first that $u\leq_R v$, then $v=uz$ is a reduced product and $w=vy=uzy$ is also a reduced product. Therefore $x=zy$ and 
$$
\varphi_w(u)=x^{-1}=y^{-1}z^{-1}\geq_R y^{-1}=\varphi_w(v). 
$$
Moreover: $d(\varphi_w(u),\varphi_w(v))=\ell(xy^{-1})=\ell(u^{-1}w\cdot w^{-1}v)=\ell(u^{-1}v)=d(u,v). $
 \end{proof}
 
 \begin{proof}[Proof of Theorem~\ref{thm:Equivalent}] The equivalence between Conjecture~\ref{conj:1} and Conjecture~\ref{conj:2} follows from Proposition~\ref{prop:rectangled}. By  Proposition~\ref{prop:IntWeak}, the left multiplication by $w^{-1}$ is an order-reversing isomorphism from $[e,w]_R$ to $[e,w^{-1}]_R$, i.e.,  $u\leq_R v$ if and only if $w^{-1}v\leq_R w^{-1}u$, for all $u,v\in [e,w]_R$. In particular,  $\coatom([e,w]_R)=\atom([e,w^{-1}]_R$ and, for any $u\leq_R w$, $\coatom([e,u]_R)=\atom([w^{-1}u,w^{-1}]_R$. So Conjecture~\ref{conj:2} is equivalent to Conjecture~\ref{conj:3}.
 \end{proof}

\section{The case of symmetric groups (Type A)}\label{se:Sn}

In  the two next sections, we provide a direct proof of Conjecture~\ref{conj:1} for Coxeter systems of type $A$ (symmetric groups) and $B$ (hyperoctahedral groups). For $n\in\mathbb N^*$, we denote by $\mathcal S_n$ the {\em symmetric group} of permutations of $\{1,\dots , n\}$  and by $W_n$ the {\em hyperoctahedral group} of signed permutations of  $\Bsetnn:=\{\pm 1,\dots,\pm n\}$.

\begin{thm}\label{thm:Main} Let $n\in\mathbb N^*$ and assume  that $W=\mathcal S_n$ or $W=W_n$. Then for any~$w\in W$ and~$\{u,v\}$ a bipartition of $w$, we have:
 $$
 d_R(w)=d_R(u)+d_R(v).
 $$
\end{thm}

The proof of the type $A$  part of Theorem~\ref{thm:Main} is given in this section  and the part regarding type $B$ is given in the next section~\S\ref{se:Bn}. 

\begin{remark}\label{rem:GeneralizeAB} Our proof does not require the use of the longest element, we do show directly Conjecture~\ref{conj:1} in type $A$ and $B$ by means of shuffle and standardization that seems to suggest that some phenomenon linked to decompositions along minimal coset representatives for spherical standard parabolic cosets might be hidden behind the proof we provide, By providing those proofs ($A$ and $B$) we hope that it might lead to an answer to Conjecture~\ref{conj:1} in general.  
\end{remark}

For $n\in\mathbb N^*$, we denote by $\Symn$ the symmetric group and $S=\{\tau_1,\dots,\tau_{n-1}\}$ the set of simple transpositions $\tau_i = (i\ i+1)$ (in the cycle notation). It is well-known that $(\Symn,S)$ is a Coxeter system of type $A_{n-1}$ in which the set of reflections is the set of all transpositions:
$$
T=\{(i\ j)\mid 1\leq i<j\leq n\},
$$
see for instance \cite[Chapters 1 and 2]{BjBr05} for more details.

A permutation  $\sigma \in \Symn$, as a bijection on the set $[n]=\{1,\dots,n\}$, is written as a word on the alphabet $[n]$: $\sigma=\sigma(1) \sigma(2) \dots \sigma(n)$ (the  {\em one-line notation}). The (left) inversion set of $\sigma$ has therefore the following well-known interpretation: 
$$
T(\sigma)=\{(\sigma(i)\ \sigma(j))\mid 1\leq i<j\leq n,\ \sigma(i)>\sigma(j) \}.
$$
Moreover the right and left descent sets are:
$$
D_L(\sigma)=\{\tau_i \mid \sigma^{-1}(i)>\sigma^{-1}(i+1)\}\textrm{ and }D_R(\sigma)=\{\tau_i \mid \sigma(i)>\sigma(i+1)\}.
$$
We say that $\sigma$ has a \emph{descent in position $i$} if $\tau_i\in D_R(\sigma)$.

\begin{ex}\label{ex:Sn1} Let $\sigma = 12684375\in\mathcal S_8$ in one-line notation. The (left) inversion set is:
$$
T(\sigma)=\{(3\ 6),\ (4\ 6),\ (5\ 6),\ (4\ 8),\ (3\ 8),\ (5\ 8),\  (7,\ 8),\ (3\ 4),\ (5\ 7)\},
$$ 
since in cycle notation $(i\ j)=(j\ i)$. 
The descent sets are $D_L(\sigma)=\{\tau_3,\tau_5,\tau_7\}$ and $D_R(\sigma)=\{\tau_4, \tau_5,\tau_7\}$. Therefore $\sigma$ has descents in positions $4,5$ and $7$.
\end{ex}

%

\subsection{Standardization of words, inversions and descents}\label{ss:TypeA-inversions}  A permutation $\sigma\in \Symn$ is, in particular, a word of length $n$ on the alphabet $\NN^*$ that contains exactly once each letter in $[n]$.  The notions of (left) inversions and right descents are naturally generalized to any word $w=w_1w_2\dots w_n$ on $\NN^*$ of length $n$ as follows:
\begin{itemize}
\item  $w$ has the {\em inversion} $(i,j)$ if $i < j$ and $w_i > w_j$;
\item $w$ has a  \emph{descent in position} $i$ if $w_i > w_{i+1}$. 
\end{itemize}

In this section we only consider {\em injective words}, i.e., words without repetition of letters. For instance $37142$ is injective while $11372142$ is not. Let $w$ be an injective word on $\NN^*$ and $a_1, a_2,\dots,a_k \in\NN^*$ $k$ letters in $w$, which are pairwise distincts since $w$ is injective. The notation
$$
w=\dots a_1\dots a_2\dots a_k\dots
$$
means that $w$ is the concatenation $u_1 a_1u_2 a_2\dots u_k a_k$ for some $u_1,u_2,\dots,u_k$ (possibly empty) words on $\NN^*$.  The set of inversions of the injective word $w$ is denoted by:
\begin{equation}\label{eq:Invs}
\invs(w) := \lbrace (w_j, w_i) \mid i < j \text{ and } w_i > w_j \rbrace= \lbrace (w_j, w_i) \mid w=\dots w_i \dots w_j \dots \rbrace.
\end{equation}

\begin{ex}\label{ex:Words1}  The word $w = 391427$ has $2$ descents in positions $2$ and $4$ and $\invs(w)=\{(1,3),\ (1,9), \ (4,9),\  (2,3),\ (2,9),\ (2,4),\ (7,9)\}$.
\end{ex}

\begin{remark} In the definition of inversions of the word $w$, we use the notation of a couple $(a,b)$ for an inversion instead of the cycle notation for which $(a\ b)=(b\ a)$.  In particular, if $w\in\Symn$, then $T(w)=\{(a\ b)\mid (a,b)\in\invs(w)\}$.

The reason is that, in the proof of Theorem~\ref{thm:Main}, we need to keep track of the following fact:  $(a,b)\in \invs(w)$ means that the letter $a$ is strictly smaller that the letter $b$ and the letter $a$ appears after the letter $b$ in $w$, i.e.,  $w=\dots b\dots a\dots$.
\end{remark}

For any word $w$ on $\NN^*$ of length $n$, the \emph{standardization} of $w$, denoted by $\std(w)$, is the unique permutation $\sigma\in \Symn$ such that for all $i < j$, $\sigma(i) > \sigma(j)$ if and only if $w_i > w_j$.  This corresponds to relabeling the letters of $w$ from $1$ to $n$ starting with the occurrences of the smallest letter from left to right, then the second smallest, etc. 

\begin{ex} Let $w =391427$ as in Example~\ref{ex:Words1}, we obtain $\std(w) = 361425$. The words $w$ and $\std(w)$ have the same number of inversions, at the same positions. In particular, the positions of the descents of $w$ are exactly the positions of the descents of $\std(w)$. They correspond to the inversions $(1,9)$ and $(2,4)$ in~$w$ and to the inversions $(1,6)$ and $(2,4)$  in $\std(w)$.
\end{ex}

The proof of Conjecture~\ref{conj:1} in type $A$ relies on the properties of  inversion sets.  We need, in order to prove Theorem~\ref{thm:Main}, to translate the notion of biclosed sets  and some of the results of Proposition~\ref{prop:Biclos} in terms of inversion sets of permutations. 

Let $I$ be a set of pairs $(a,b)$ with $1 \leq a < b \leq n$ for some $n \geq 1$. The set $I$ is {\em transitive} if for any $(a,b),(b,c) \in I$, we have $(a,c) \in I$.  Denote $I^c = \lbrace (a,b); 1 \leq a < b \leq n, (a,b) \not\in I \rbrace$ the complement of $I$. The well-known {\em transitivity property of inversion sets} is:  $I$ is the inversion set of a permutation $w$ of $\Symn$ if and only if both $I$ and $I^c$  are transitive sets.  There is an alternate useful way to verify the transitivity of $I^c$, as stated in the next proposition. 

\begin{prop}
\label{prop:invs}
Let $I$ be a set of tuples $(a,b)$ with $1 \leq a < b \leq n$ for some $n \geq 1$. Then $I$ is the inversion set of a certain permutation of $\Symn$ if and only if it satisfies the following conditions for all $1 \leq a < b < c \leq n$

\begin{enumerate}
\item $(a,b) \in I$ and $(b,c) \in I$ implies that $(a,c) \in I$;
\label{prop-item:invs-trans}
\item $(a,c) \in I$ implies that either $(a,b) \in I$, or $(b,c) \in I$.
\label{prop-item:invs-planar}
\end{enumerate}
\end{prop}


\subsection{Bipartition, recursive decomposition of a permutation and proof of Theorem~\ref{thm:Main} in the case of symmetric groups}

Before getting into the core of the proof of Theorem~\ref{thm:Main}, we give an example of a bipartition of $\Symn$ within the language of the preceding section. We use this example to illustrate the steps of our proof.

\begin{ex}[An example of a bipartition in type $A$]\label{ex:TypeA} 
Let $w = 526341$. It has $10$ inversions : $(1,2), (1,3), (1,4), (1,5), (1,6), (2,5), (3,5), (3,6), (4,5), (4,6)$ and $3$ descents in positions $1$, $3$, and $5$. Then $u = 234561$ and $v = 152634$ form a bipartition of $w$. Indeed, the inversions of $u$ are $(1,2),  (1,3), (1,4), (1,5), (1,6)$ while the inversions of $v$ are $(2,5), (3,5), (3,6), (4,5), (4,6)$.

We check that Conjecture~\ref{conj:1} is indeed satisfied as there is $1$ descent in $u$ (in position $5$) and $2$ descents in $v$ (in positions $2$ and $4$). 
\end{ex}

We are now ready to introduce  the main notions that are needed to prove Theorem~\ref{thm:Main}. 

\begin{defi}
Let $w$ be a permutation of $\Symn$ with $n \geq 1$. The \emph{decreased word of}  $w$ is obtained by removing the letter $n$ from $w$ and is denoted by $\dec{w}$.
\end{defi}

For an example, if $w = 526341$, then $\dec{w} = 52341$.

\begin{lem}
\label{lem:dec}
Let $\lbrace u, v \rbrace$ be a bipartition of a permutation $w \in \Symn$ with $n \geq 1$, then $\lbrace \dec{u}, \dec{v} \rbrace$ is a bipartition of $\dec{w}$.
\end{lem}

\begin{proof}
Note that for $w \in S_n$ with $n \geq 1$, then $\invs(\dec{w}) = \lbrace (a,b) \in \invs(w); b \neq n \rbrace$. In other words, we remove the inversions containing $n$. In particular, as $\invs(w) = \invs(u) \sqcup \invs(v)$ we have $\invs(\dec{w}) = \invs(\dec{u}) \sqcup \invs(\dec{v})$, which concludes the proof.
\end{proof}

Let $w$ be a word on $\NN$ and $I \subset \NN$. We write $\rest{w}{I}$ the subword of $w$ containing only the letters in $I$. If $w$ is a permutation, then $\rest{w}{I}$ is not in general a permutation, but $\srest{w}{I}$ is. For an example, if $w = 234561$ and $I = \lbrace 1, 3, 4, 6 \rbrace$, then $\rest{w}{I} = 3461$ and $\srest{w}{I} = 2341$.

\begin{defi}
\label{def:decomp}
Let $\lbrace u, v \rbrace$ be a bipartition of a permutation $w \in \Symn$. The \emph{right letters} of $w$ are defined as the set of letters in $w$ at the right-hand side of the letter $n$:

\begin{equation*}
\R (w) := \lbrace a\mid (a,n) \in \invs(w) \rbrace \cup \lbrace n \rbrace
\end{equation*}
The \emph{left letters} of $w$ are the complement in $[n]$ of the right letters, i.e., the set of letters in $w$ at the left-hand side of the letter $n$:

\begin{equation*}
\L(w) := \lbrace b\mid 1 \leq b < n, (b,n) \not\in \invs(w) \rbrace.
\end{equation*}
The {\em left} (resp. {\em right}) {\em words} of $w$, $u$, and $v$ are respectively:
\begin{align*}
w_\L &= \srest{w}{\L(w)}, & w_\R &= \srest{w}{\R(w)} \\
u_\L &= \srest{u}{\L(w)}, & u_\R &= \srest{u}{\R(w)} \\
v_\L &= \srest{v}{\L(w)}, & v_\R &= \srest{v}{\R(w)}. 
\end{align*}
\end{defi}

In other words, we \emph{cut} $w$ into two parts (left and right) just before the maximal letter. Then we select those letters in $u$ and $v$ and standardize to form new permutations. Note that $n$ is by definition a \emph{right} letter. The set $\L(w)$ of left letters might be empty, in this case the left words are the empty word and the right words are the original words $u,v,w$. If $n \geq 1$, the set $\R(w)$ of right letters is never the empty set as it contains at least the letter $n$.  Note also that the left (resp. right) words of $u$ and $v$ depend on the left (resp. right) letters of $w$.

\begin{ex}[Continuation of Example~\ref{ex:TypeA}] The left letters of $w = 526341$ are $\lbrace 2, 5 \rbrace$ and the right letters are $\lbrace 1, 3, 4, 6 \rbrace$. We then have

\begin{align*}
w_\L &= \std(52) = 21, & w_\R &= \std(6341) = 4231, \\
u_\L &= \std(25) = 12, & u_\R &= \std(3461) = 2341,  \\
v_\L &= \std(52) = 21, & v_\R &= \std(1634) = 1423. 
\end{align*}
\end{ex}

The proof of Theorem~\ref{thm:Main} in the case of symmetric groups relies on the two following propositions, which are proven in the following sections.

\begin{prop}
\label{prop:decomposition}
Let $\lbrace u, v \rbrace$ be a bipartition of a permutation $w \in \Symn$. Then the left words $\lbrace u_\L, v_\L \rbrace$ form a bipartition of $w_\L$ and the right words $\lbrace u_\R, v_\R \rbrace$ form a bipartition of $w_\R$ and

\begin{align}
\label{eq:deswdecomp}
\d(w) &= \d(w_\L) + \d(w_\R), \\
\label{eq:desudecomp}
\d(u) &= \d(u_\L) + \d(u_\R), \\
\label{eq:desvdecomp}
\d(v) &= \d(v_\L) + \d(v_\R).
\end{align}

\end{prop}

\begin{prop}
\label{prop:decreasing}
Let $\lbrace u, v \rbrace$ be a bipartition of a permutation $w \in \Symn$ such that $n \geq 2$ and $w_n \neq n$. Then the right words $u_\R$ and $v_\R$ satisfy the following equation:

\begin{equation}
\label{eq:decreasing}
\d(\dec{u_\R}) + \d(\dec{v_\R}) = \d(u_\R) + \d(v_\R) -1.
\end{equation}
\end{prop}

\begin{proof}[Proof of Theorem~\ref{thm:Main} for $W=\mathcal S_n$] The proof is  by induction on $n\in\mathbb N^*$. The result is obviously true in $\Sym{1}$ that contains only the identity with 0 descents and inversions. Now, let $\lbrace u, v \rbrace$ be a bipartition of a permutation $w \in \Symn$ with $n \geq 2$. We suppose that the statement of  Theorem~\ref{thm:Main} is satisfied for all $1 \leq k < n$. Write $w=w_1\dots w_n$ in one-line notation. 

If $w_n = n$, then $\invs(w)$ does not contain any inversion $(*, n)$. It is also true for $\invs(u)$ and $\invs(v)$ since $\{u,v\}$ is a bipartition of $n$. Therefore $u_n = v_n =n$. In particular, there is no descent in position $n-1$ for neither $u$, $v$, nor $w$. So $w$, $u$ and $v$ are in $\mathcal S_{n-1}$ viewed as a subgroup of $\Symn$.  By Lemma~\ref{lem:dec}, $\lbrace \dec{u}, \dec{v} \rbrace$ is a bipartition of $\dec{w}$. So by induction $\d(w) = \d(\dec{w}) = \d(\dec{u}) + \d(\dec{v}) = \d(u) + \d(v)$.

Now suppose that $w_n \neq n$. This means that $w$ is an element of~$\Symn$ that is not contained in any $\Sym{k}$ with $k < n$ viewed as a subgroup of $\Symn$. We consider the decomposition given by Definition~\ref{def:decomp}. Let $k<n$ be the position of the letter $n$. We have two cases.

\paragraph{Case $k \neq 1$.} This implies that $w_\L$, $u_\L$ and $v_\L$ are not empty words. Therefore, by Proposition~\ref{prop:decomposition}, $\lbrace u_\L, v_\L \rbrace$ is a bipartition of $w_\L$ with $w_\L \in \Sym{k-1}$, and~$\lbrace u_\R, v_\R \rbrace$ is a bipartition of $w_\R$ with $w_\R \in \Sym{n-k+1}$. We have~$k-1 < n$ and  $n-k+1 <n$, since~$1<k < n$,. By induction, we obtain $\d(w_\L) = \d(u_\L) + \d(v_\L)$ and $\d(w_\R) = \d(u_\R) + \d(v_\R)$. Using Eq.~\eqref{eq:deswdecomp}, Eq.~\eqref{eq:desudecomp}, and~Eq.~\eqref{eq:desvdecomp} of Proposition~\ref{prop:decomposition}, we obtain that $\d(w) = \d(u) + \d(v)$.

\paragraph{Case $k=1$.} This implies that the left words are empty, \emph{i.e.}, $w = w_\R$, $u= u_\R$, and $v = v_\R$ and we cannot use directly the induction hypothesis. Nevertheless, we now have $w_1 = n$. In particular, as $n \geq 2$, this means that there is a descent in position $1$ in $w$ which disappears if we remove $n$. In other words, $\d(\dec{w}) = \d(w) - 1$.  Moreover, by Proposition~\ref{prop:decreasing}, we have $\d(\dec{u}) + \d(\dec{v}) = \d(u) + \d(v) -1$. By Lemma~\ref{lem:dec}, we know that $\lbrace \dec{u}, \dec{v} \rbrace$ is a bipartition of $\dec{w}\in\mathcal S_{n-1}$. Therefore, by induction,  $\d(\dec{w}) = \d(\dec{u}) + \d(\dec{v})$. We conclude that:
$$
\d(w)=\d(\dec{w}) +1=\d(\dec{u}) + \d(\dec{v}) +1= \d(u) + \d(v).
$$
\end{proof}

\begin{remark} The definition of the left, right and decreased words of $w$ by choosing $n$ as the letter is crucial in the above proof. A more `Coxeter-like' strategy to prove Theorem~\ref{thm:Main} by induction we tried (and failed) is the following:  If $D_R(w)=S$, then $w=w_\circ$ is the longest element and we can therefore conclude by Proposition~\ref{prop:Longest}. Now assume that $\tau_i\in S\setminus D_R(w)$ and define  the left, right and decreased words accordingly to $w_{i+1}$. Unfortunately, then, the analog of Proposition~\ref{prop:decomposition} fails to be true.

 For instance, consider $w=54231$ with the bipartition $\{ u=23514,v=41523\}$. Take $\tau_3\in S\setminus D_R(w)$, then $w_4=3$, $\L(w)=\{542\}$ and $\R(w)=\{1,3\}$. With this definition depending of $3$ instead of $n=5$, we obtain $u_\L = 132$, $u_\R=21$, $v_\L=312$ and $v_R=12$. Observe that $\d(u)=1<2=d_R(u_\L)+d_R(u_\R)$, which means that Proposition~\ref{prop:decomposition} fails to be true in this case. 
\end{remark}

\subsection{Proof of Proposition~\ref{prop:decomposition}}

We start with a simple, yet fundamental, lemma.

\begin{lem}
\label{lem:right-left-inversions}\label{rem:n-after}
Let $\lbrace u, v \rbrace$ be a bipartition of a permutation $w \in \Symn$. Then there is no inversion $(a,c)$ of $u$ or $v$ such that $a\in \L(w)$ is a left letter of $w$ and $c\in\R(w)$ is a right letter of~$w$.
In particular:
\begin{itemize}
\item If $a\in\R(w)$ and $c \in \L(w)$ such that 
$$
u=\dots a\dots c\dots\quad \textrm{or}\quad v=\dots a\dots c\dots,
$$ 
then $a<c$.
\item The letter $n$ is placed after all the left letters in both $u$ and $v$.
\end{itemize}
\end{lem} 

\begin{proof}
This is readily seen  on the word $w$ as all left letters are indeed at the left-hand side of all the right letters of $w$, which are at the right-hand side of the maximal letter $n$. So $w$ does not contain an inversion $(a,c)$ with $a \in \L(w)$ and $c \in \R(w)$. The lemma follows since inversions of $u$ and $v$ are contained  inversions of $w$, by definition of bipartitions.
\end{proof}

\begin{remark} A more ¨Coxeter-like' argument to prove this lemma is with the transitivity of inversions (bipodal sets). Indeed, as $c\in\R(w)$, then by Definition~\ref{def:decomp}, there is an inversion $(c,n)$ in $w$. If $(a,c)$ is also an inversion of $w$, we obtain that $(a,n)$ is an inversion by transitivity and so $a\in\R(w)$.   
\end{remark}

\begin{ex}[Continuation of Example~\ref{ex:TypeA}]  We have $u = 234561$ with $2$ and $5$ being the only left letters. We note that the two right letters $3$ and $4$ are placed before the letter $5$ and are indeed smaller.
\end{ex}

\begin{prop}
\label{prop:decents}
Let $\lbrace u, v \rbrace$ be a bipartition of a permutation $w \in \Symn$. If there is a descent in $u$ (or in $v$) between two letters $c$ and $a$, then either $a,c\in\R(w)$ or $a,c\in\L(w)$.
\end{prop} 

\begin{ex}[Continuation of Example~\ref{ex:TypeA}]   We have $u = 234561$ and $v = 152634$. We mark the descents in $u$ with a dot for clarity:  $u = 23456.1$ has a single descent between $6$ and $1$, both right letters. For  $v = 15.26.34$ we have one descent between $5$ and $2$ (both left letters) and one descent between $6$ and $3$ (both right letters).
\end{ex}

\begin{proof} Suppose that we have a descent in $u$ between two letters $c$ and $a$ (with $a < c$). So $u=\dots ca\dots$.  Note that in particular, $u$ contains the inversion $(a,c)$. So by Lemma~\ref{lem:right-left-inversions}, if $c\in\R(w)$, then  $a\in\R(w)$.

Suppose by contradiction that $c\in\L(w)$  and $a\in\R(w)$. Since $c\in\L(w)$, we have $u=\dots ca \dots n\dots$ by Lemma~\ref{rem:n-after}. In particular, $(a,n) \not\in \invs(u)$. Since $a\in\R(w)$, we know that $(a,n)\in\invs(w)$. Moreover, $\invs(w) = \invs(u) \sqcup \invs(v)$ by definition of bipartition, so $(a,n) \in \invs(v)$ and $(a,c)\notin\invs(v)$.  By Proposition~\ref{prop:invs}~\eqref{prop-item:invs-planar} and since $a < c < n$, we must have  $(c,n) \in \invs(v) \subset \invs(w)$. So $c\in\R(w)$ by Lemma~\ref{lem:right-left-inversions},   which is a contradiction since $c\in\L(w)$.
\end{proof}

\begin{proof}[Proof of Proposition~\ref{prop:decomposition}]
Let $\lbrace u, v \rbrace$ be a bipartition of a permutation $w \in \Symn$. Consider the set of left inversions of $w$:
$$
\invs_\L(w)= \lbrace (a,c) \in \invs(w)\mid a \in \L(w) \text{ and } b \in \L(w) \rbrace,
$$
\emph{i.e.}, the inversions of $w$ where both letters are left letters. Similarly, we define $\invs_\R(w)$ the set of right inversions of $w$, i.e., the inversions where both letters are right letters. As $\lbrace u, v \rbrace$ is a bipartition of $w$ and by definition of  $\L(w)$ and $\R(w)$, we have:
\begin{align*}
\invs_\L(w) &= \invs_\L(u) \sqcup \invs_\L(v), \\
\invs_\R(w) &= \invs_\R(u) \sqcup \invs_\R(v).
\end{align*}
Now the inversions of $u_\L$, $v_\L$, and $w_\L$ (resp. $u_\R$, $v_\R$, and $w_\R$) are just a relabeling of the left (resp. right) inversions of $u$, $v$, and $w$. This proves that $\lbrace u_\L, v_\L \rbrace$ is a bipartition of $w_\L$ and similarly, $\lbrace u_\R, v_\R \rbrace$ is a bipartition of $w_\R$, which proves the first statement of the proposition.

Equality~\eqref{eq:deswdecomp} is readily seen: $w_\L$ is the (standardized) left part of $w$ while $w_\R$ is the (standardized) right part of $w$. If $n$ is in position $j$, all descents in positions $i < j-1$ are descents in $w_\L$ and all descents in positions $k \geq j$ are descents in $w_\R$. Moreover, there is no decent in position $j-1$ because the letter at this position is strictly smaller than $n$, the letter at position $j$.

The permutation $u$ and $v$ play similar role so one just has to show~\eqref{eq:desudecomp} without loss of generality. By Proposition~\ref{prop:decents} we have:
\begin{equation*}
\d(u) \leq \d(u_\L) + \d(u_\R).
\end{equation*} 
Indeed, if there is a descent in $u$ between $c$ and $a$, the descent is kept either on~$u_L$ or~$u_R$ between the standardized equivalent of $c$ and $a$. We need to prove that no new descent is created to show that~\eqref{eq:desudecomp} is satisfied. 

Suppose first that there is a descent in $u_\L$ that does not arise from a descent of~$u$. So there is two letters $c > a$ in $\L(w)$ which are not consecutive in $u$, but their standardized equivalent in $u_\L$ are. So $u=\dots c d_1d_2\dots d_k a\dots$, where $d_1d_2\dots d_k$ is a non empty word on $\mathbb N^*$ constituted only of letters  $d_i\in\R(w)$. By Proposition~\ref{prop:decents}, there no descent between $c$ and $d_1$ and so $d_1 > c$. Therefore $(a,d_1)\in\invs(u)$ with $a\in \L(w)$ and $d_1\in\R(w)$, contradicting  Lemma~\ref{lem:right-left-inversions}. So there is no new descent in $u_\L$.

Suppose now that there is a descent in $u_\R$ that does not arise from a descent of $u$. So there is two letters $c > b$ in $\R(w)$ which are not consecutive in $u$, but their standardized equivalent in $u_\R$ are. So $u=\dots c a_1a_2\dots a_k b\dots$, where $a_1a_2\dots a_k$ is a non empty word on $\mathbb N^*$ constituted only of letters  $a_i\in\L(w)$. By a similar reasoning as above, we have $a_k < b<c$. So $(a_k,c)\in\invs(u)$ with $a_k\in\L(w)$ and $c\in\R(w)$, which contradicts again Lemma~\ref{lem:right-left-inversions}. 

In conclusion, there is no descent that are created in $u_\L$ or $u_\R$ that does not arise from the descents of $u$, hence we have $\d(u) = \d(u_\L) + \d(u_\R).$
\end{proof}

\subsection{Proof of Proposition~\ref{prop:decreasing}} First we illustrate the proof with the continuation of Example~\ref{ex:TypeA}. We have:

$$
w_\R = \std(6341) = 4231, \, u_\R = \std(3461) = 2341, \, v_\R = \std(1634) = 1423.
$$
When we \emph{decrease} the permutations by removing the letter $4$, we obtain:
$$
\dec{w_\R} = 231, \,
\dec{u_\R} = 231, \,
\dec{v_\R} = 123.
$$
We lose one descent in $w_\R$ as $4$ is the first letter and maximal by definition of the right word. In $u_\R$, the descent between $4$ and $1$ is `kept' in $\dec{u_\R}$ because the letter before $4$ is $3 > 1$. In $v_\R$, the descent between $4$ and $2$ is `lost' because the letter before $4$ is $1 < 2$.

This phenomenon is at the core of the proof of Proposition~\ref{prop:decreasing}. As the other descents are not impacted by taking the decreasing word, we obtain indeed~Eq.\eqref{eq:decreasing}. Let us start with the following simple observation.

\begin{lem}
\label{lem:right-word-letters}
Let $\lbrace u, v \rbrace$ be a bipartition of a permutation $w \in S_n$ with $w_1 = n$. For any letter $a\in[n-1]$, we have $u=\dots a \dots n\dots$ if and only if $v=\dots n\dots a\dots$.
\end{lem}

In the previous example, with $\lbrace 2341, 1423 \rbrace$ a bipartition of $4231$, you see that the letters~$\lbrace 2, 3 \rbrace$ are before~$4$ in~$u$ and after~$4$ in~$v$, while~$1$ is after~$4$ in~$u$ and before~$4$ in~$v$. 

\begin{proof}
As $w_1 = n$, we know that $(a,n) \in \invs(w)$. Assume $u=\dots a \dots n\dots$. Then $(a,n) \not\in \invs(u)$. By definition of bipartitions, $\invs(w) = \invs(u) \cup \invs(v)$ so $(a,n) \in \invs(v)$ which means that $v=\dots n\dots a\dots$. The convers follows since the role of $u$ and $v$ are symmetric in the preceding argument. 
\end{proof}

\begin{proof}[Proof of Proposition~\ref{prop:decreasing}]
Let $\lbrace u, v \rbrace$ be a bipartition of a permutation $w \in S_n$ with $w_1 = n$. We prove that $\d(\dec{u}) + \d(\dec{v}) = \d(u) + \d(v) +1$. Note that this proves the proposition in general, as if $\lbrace u, v \rbrace$ is a bipartition of a permutation $w \in S_n$, then $\lbrace u_\R, v_\R \rbrace$ is a bipartition of a certain permutation $w_\R \in S_k$ with $k \leq n$ such that ${w_\R}(1) = k$. When the permutation $u$ is decreased to $\dec{u}$, all the descents of $u$ are left unchanged except the possible descent created by the letter $n$. 

Suppose first that $n$ is the last letter of $u$, then $\d(\dec{u}) = \d(u)$. By Lemma~\ref{lem:right-word-letters}, as all letters are the lefthand side of $n$ in $u$, they are therefore all at the right-hand side of $n$ in $v$ and $v_1 = n$. In particular $\d(\dec{v}) = \d(v) -1$, which implies the desired result. Similarly, if $u_1 =n$, then $v_n = n$ and we have $\d(\dec{u}) = \d(u) - 1$ and $\d(\dec{v}) = \d(v)$.

Now suppose that $n$ is placed in $u$ in a position that is neither $1$ nor $n$, which implies that it is also the case in $v$ by Lemma~\ref{lem:right-word-letters}. So  there is two letters $a<b$ such that:
\begin{itemize}
\item either $u = \dots bna \dots$ and $\d(\dec{u}) = \d(u)$; the descent is said to be {\em kept} in~$\dec{u}$. 
\item Or we have $u = \dots anb \dots$ and $\d(\dec{u}) = \d(u) - 1$, the descent is said to be {\em lost} in $\dec{u}$. 
\end{itemize}
If the descent is kept in $\dec{u}$ and lost in $\dec{v}$ (or the other way around), then~\eqref{eq:decreasing} is true. We prove that the descent cannot be kept in both, nor lost in both. We proceed by contradiction.

\paragraph{Suppose that the descent is kept in both $\dec{u}$ and $\dec{v}$.} Let  $a' < b'$ such that $v=\dots b'na'\dots$ (the descent is kept in $\dec{v}$). Since  $a$ is at the right-hand side  of $n$ in $u$, the letter $a$ is at the left-hand side of $n$ in $v$, by Lemma~\ref{lem:right-word-letters}. In particular $a \neq a'$. Similarly with  $b$, $a'$, and $b'$, we obtain that $u$ and $v$ are of the form:
\begin{align*}
u &= \dots a' \dots bna \dots b'  \dots, \\
v &= \dots a \dots b' n a' \dots b \dots.
\end{align*}
Note that we might have $a' = b$ or $a = b'$. Without loss of generality, we suppose $a' < a$; so in particular $a' < a < b$ (the descent is lost in $\dec{v}$). We have $(a', a) \in \invs(v) \subseteq \invs(w)$. Since  $(a,b) \in \invs(u) \subseteq \invs(w)$, we have by transitivity of the inversion set $\invs(w)$ (Proposition~\ref{prop:invs}~(\ref{prop-item:invs-trans})) that $(a', b) \in \invs(w)$. But $(a', b)\notin\invs(u)\cup\invs(v)=\invs(w)$, which is a contradiction.

\paragraph{Suppose that the descent is lost in both $\dec{u}$ and $\dec{v}$.} Let  $a' < b'$ such that $v=\dots a'nb'\dots$. Similarly as in the preceding case,  we obtain that $u$ and $v$ are of the form:
\begin{align*}
u &= \dots b' \dots anb \dots a'  \dots, \\
v &= \dots b \dots a' n b' \dots a \dots.
\end{align*}
Without loss of generality, we suppose $a' < a$. We obtain $a' < a < b$. Therefore $(a', b) \in \invs(u) \cap \invs(v)$, contradicting that $\invs(u) \cap \invs(v) = \emptyset$ ($\{u,v\}$ is bipartition of $w$).
\end{proof}

\section{The case of hyperoctahedral groups (type $B$)}\label{se:Bn}

For $n\in\mathbb N_{\geq 2}$, consider now the Coxeter system   $(W_n, S)$ of type $B_n$, i.e.,  the set of simple reflections is $S= \lbrace \tau_0, \tau_1, \dots, \tau_{n-1} \rbrace$ and the Coxeter graph of $(W_n,S)$ is: 
\begin{center}
\begin{tikzpicture}[sommet/.style={inner sep=2pt,circle,draw=blue!75!black,fill=blue!40,thick}]
	\node[sommet,label=below:$\tau_0$] (a0) at (0,0) {};
	\node[sommet,label=below:$\tau_1$] (a1) at (1,0) {} edge[thick] node[auto,swap] {4} (a0);
	\node[sommet,label=below:$\tau_2$] (a2) at (2,0) {} edge[thick] node[auto,swap] {} (a1);
	\node[] (a3) at (2.7,0) {} edge[thick] node[auto,swap] {} (a2);
	\node[] (a4) at (3.5,0) {} edge[thick,dotted] node[auto,swap] {} (a3);
	\node[sommet,label=below:$\tau_{n-1}$] (a5) at (4.2,0) {} edge[thick] node[auto,swap] {} (a4);
\end{tikzpicture}
\end{center}
In particular, the standard parabolic subgroup $W_{\{\tau_1,\dots,\tau_{n-1}\}}$ is isomorphic to the symmetric group $\mathcal S_n$. 

The group $W_n$ is known as the {\em hyperoctahedral group of rank $n$} and its elements are {\em signed permutations}. More precisely, write:
$$
\Bsetnn=\lbrace \bar n, \dots, \bar 1 \rbrace \cup \lbrace 1, \dots, n \rbrace,
$$ 
where  $\bar{i}:=-i$. Then $W_n=\{\sigma \in \mathcal S_{\Bsetnn}\mid \sigma(\bar i) = \overline{\sigma(i)}\}$ is a subgroup of  $\mathcal S_{\Bsetnn}$, the group of permutations on $\Bsetnn$. The simple reflection $\tau_0$ is the transposition $(\bar 1\ 1)\in \mathcal S_{\Bsetnn}$ and, for $i>0$, the simple reflection $\tau_i$ is the product of two simple transpositions in $\Bsetnn$: $\tau_i= (\overline{i+1}\ \bar i)(i\ i+1)$. It is well-known that the set of reflections of $(W_n,S)$ is:
$$
T=\{(\bar i \ i )\mid 1\leq i\leq n\}\sqcup\{  ( \bar i\ \overline{j})(i\ j) ,\ (\bar i\ j)(i\ \overline{j})\mid 1\leq i< j\leq n \}.
$$
In particular, the subset $\{  ( \bar i\ \overline{j})(i\ j),  \mid 1\leq i< j\leq n \}$ corresponds to the set of reflections of $W_{\{\tau_1,\dots,\tau_{n-1}\}}\simeq\mathcal S_n$. 
\smallskip

A signed permutation $\sigma \in W_n$, as a permutation on the set  $\Bsetnn$, has a one-line notation
$$
\sigma=\sigma_{\bar n}\dots \sigma_{\bar 2}\sigma_{\bar 1}| \sigma_1 \sigma_2\dots \sigma_n,
$$
with the property that $\bar \sigma_i = \sigma_{\bar i}$.

\begin{remark} We have of course $\overline{\bar i} = -(-i)=i$.  The symbol `$|$' in the notation above is used to mark the symmetry $\sigma_{\bar i}=\sigma(\bar i)=\sigma(-i)=-\sigma(i)=\overline{\sigma(i)}=\overline{\sigma_i}$). Note that even though it is enough to keep track of only the positive positions to retrieve the whole element, it is more convenient in this article to keep track of the negative positions as well.
\end{remark}

The right multiplication by the simple reflections is then given by:
\begin{align*}
\sigma\tau_0=&\  \sigma_{\bar n}\dots \sigma_{\bar 2}\sigma_{1}| \sigma_{\bar 1} \sigma_2\dots \sigma_n,\\
\sigma\tau_i=&\ \sigma_{\bar n}\dots\sigma_{\bar i} \sigma_{\overline{i+1}}\dots \sigma_{\bar 1}| \sigma_{1} \dots \sigma_{i+1}\sigma_i\dots \sigma_n,
\end{align*}
where  $1\leq i<n$. 

\begin{ex}\label{ex:TypeB1} Consider the following element of $W_6$:
$$
\sigma=4 \bar{5} \bar{2} \bar{6} 3 1 | \bar{1} \bar{3} 6 2 5 \bar{4}
$$
The one-line notation  means:
$$
\sigma (-6) = \sigma(\bar 6) = 4, \sigma(-5 )= \sigma(\bar 5) = \bar 5=-5, \dots,\sigma(\bar 1)= \sigma(-1) = 1, \sigma(1) = -1=\bar 1, 
$$ 
and so on. Since $ \sigma(\bar 6) = 4$, we also have $ \sigma(6) = \bar 4$.  A reduced word for $\sigma$ is: $\tau_0 \tau_2 \tau_1 \tau_0 \tau_3 \tau_2 \tau_1 \tau_0 \tau_2 \tau_1 \tau_2 \tau_3 \tau_4 \tau_5 \tau_4 \tau_3$.
\end{ex}

\subsection{Type $B$-permutations, inversions and standardization} We refer the reader to \cite[\S8.1]{BjBr05} for more details on inversions\footnote{In \cite{BjBr05}, the authors consider left inversion sets, which are the left inversion sets under taking inverses} and descents of type $B$. With the natural total order $\bar n<\overline{n-1}<\dots \bar 1<1<2\dots <n$, the (left) inversion set of $\sigma\in W_n$ has the following well-known interpretation:
\begin{eqnarray*}
T(\sigma)&=&\{(\sigma_{\bar i} \ \sigma_i )\mid 1\leq i\leq n,\ \sigma_{\bar i} > 0>\sigma_i \}\\
&&\quad \sqcup\ \{  (\sigma_{\bar i}\ \sigma_{\bar j})(\sigma_i\ \sigma_j) \mid 1\leq i< j\leq n,  \  \sigma_i>\sigma_j\textrm{ or }\sigma_{\bar i}>\sigma_j \}.
\end{eqnarray*}
For the proof that Conjecture~\ref{conj:1} holds in type $B$, we need the interpretation of {\em inversions} and {\em descents} on injective words over the alphabet $\Bsetnn$ as discussed in \S\ref{ss:TypeA-inversions} for the alphabet $[n]$. More precisely, any element $\sigma\in W_n$ is an element of  the symmetric group $\mathcal S_{\Bsetnn}$ and, therefore, has a set of inversion $\invs(\sigma)$  as an injective words  on the alphabet $\Bsetnn$ (see Eq.(\ref{eq:Invs})):
$$
\invs(\sigma)=\{(\sigma_i, \sigma_j)\mid  i,j\in \Bsetnn,\  i<j,\ \sigma_i > \sigma_j \}.
$$
It turns out that, for $1\leq i< j\leq n$:
\begin{itemize}
\item $(\sigma_{\bar i}, \sigma_i )\in \invs(\sigma)\iff(\sigma_{\bar i} \ \sigma_i )\in T(\sigma)$;
\item $(\sigma_{\bar j}\ \sigma_{\bar i})\in \invs(\sigma)\iff (\sigma_i\ \sigma_j)\in\invs(\sigma)\iff  (\sigma_{\bar i}\ \sigma_{\bar j})(\sigma_i\ \sigma_j)\in T(\sigma)$. 
\end{itemize}
Therefore:
\begin{equation}\label{eq:InvTypeB}
T(\sigma)=\{(a\ |a|) \mid  (a,|a|)\in \invs(\sigma)\}\sqcup\{(\bar a\ \bar b)(a\ b)\mid (a\ b)\in \invs(\sigma)\}.
\end{equation}
For simplicity, we say that $(a,b)\in\invs(\sigma)$ is {\em an inversion of $\sigma$}. Of course, as transpositions, $(\bar a\ \bar b)=(\bar b \ \bar a)$, but $(\bar a, \bar b)\not =(\bar b,\bar a))$ with the notations used for $\invs(\sigma)$. The right descent sets is:
$$
D_R(\sigma)= \{\tau_0 \textrm{ if } \sigma_1<0\} \sqcup \{\tau_i \mid 1\leq i < n,\ \sigma_i >\sigma_{i+1}\}.
$$ 

\begin{ex}[Continuation of Example~\ref{ex:TypeB1}]\label{ex:TypeB2} The element
$$
\sigma=4 \bar{5} \bar{2} \bar{6} 3 1 | \bar{1} \bar{3} 6 2 5 \bar{4}
$$
has right descent set $D_R(\sigma)=\{\tau_0,\tau_1,\tau_3,\tau_5\}$ and $d_R(\sigma)=4$. We mark the descents of $\sigma$ by a `dot' in the lone-line notation:
\begin{equation*}
4  \bar{5} \bar{2}  \bar{6} 3 1 | . \bar{1} . \bar{3} 6 . 2 5 . \bar{4}
\end{equation*}
The inversions of $\invs(\sigma)$ are:

\renewcommand{\arraystretch}{1.5}
\begin{center}
\begin{tabular}{|cc|c|cc|c|}
\hline
$\invs(\sigma)$&                    &$T(\sigma)$& $\invs(\sigma)$&         &$T(\sigma)$\\
\hline        
$(\bar{1}, 1)$       &                 &     $(\bar{1}\ 1)$             &$(\bar{4}, 1)$,       & $(\bar{1}, 4)$& $(\bar{4}\ 1)(\bar{1}\ 4)$\\
$(\bar{3}, 3)$       &                 &    $(\bar{3}\ 3) $           &$(\bar{4}, 3)$,       & $(\bar{3}, 4)$&$(\bar{4}\ 3)(\bar{3}\ 4)$ \\
$(\bar{4}, 4)$       &                 &   $ (\bar{4}\ 4)$             &$(\bar{4}, \bar{2})$, & $(2, 4)$& $(\bar{4}\ \bar{2})(2\ 4)$\\
$(\bar{4}, 5)$,       & $(\bar{5}, 4)$  & $(\bar{4}\ 5)(\bar{5}\ 4)$            & $(5, 6)$,             & $(\bar{6}, \bar{5})$&$(5\ 6)(\bar{6}\ \bar{5})$ \\
$(\bar{4}, 2)$,       & $(\bar{2}, 4)$ &  $(\bar{4}\ 2)(\bar{2}\ 4)$            & $(2, 6)$,             & $(\bar{6}, \bar{2})$&$(2\ 6)(\bar{6}\ \bar{2})$ \\
$(\bar{4}, 6)$,       & $(\bar{6}, 4)$ &   $(\bar{4}\ 6)(\bar{6}\ 4)$             & $(2, 3)$,             & $(\bar{3}, \bar{2})$&$(2\ 3)(\bar{3}\ \bar{2})$ \\
$(\bar{4}, \bar{3})$, & $(3, 4)$       &   $(\bar{4}\ \bar{3})(3\ 4)$               &$(\bar{3}, \bar{1})$, & $(1, 3)$&$(\bar{3}\ \bar{1})(1\ 3)$ \\
$(\bar{4}, \bar{1})$ ,& $(1, 4)$       &   $(\bar{4}\ \bar{1})(1\ 4)$                 & $(\bar{3}, 1)$,       & $(\bar{1}, 3)$&$(\bar{3}\ 1)(\bar{1}\ 3)$ \\
\hline
\end{tabular}.
\end{center}
The set $\invs(\sigma)$ has cardinality $29$ whereas $\ell(\sigma)=16=|T(\sigma)|$.  
\end{ex}

Even though the sets $\invs(\sigma)$ and $T(\sigma)$ are different, we obtain directly from Eq.(\ref{eq:InvTypeB}) the following characterization of bipartitions of elements of $W_n$.

\begin{lem}
\label{lem:B-invsT}
Let $w$ be an element of $W_n$, then $\lbrace u, v \rbrace$ forms a bipartition of $\sigma$ if and only if $\invs(w) = \invs(u)  \sqcup \invs(v)$.
\end{lem}

We now state the equivalent of Proposition~\ref{prop:invs} for type $B$ elements. This is a direct translation of the biclosed condition on Coxeter groups inversion sets and also admits an elementary proof.

\begin{prop}
\label{prop:invsB}
Let $I$ be a set of tuples $(a,b)$ with $a,b \in \Bsetnn$ for some $n \geq 1$ and $a < b$. Then $I = \invs(\sigma)$ of a certain word $\sigma$ corresponding to an element of $W_n$ if and only if it satisfies the following for all $a,b,c \in \Bsetnn$ with $a < b < c$

\begin{enumerate}
\item $(a,b) \in I$ and $(b,c) \in I$ implies that $(a,c) \in I$;
\label{prop-item:invsB-trans}
\item $(a,c) \in I$ implies that either $(a,b) \in I$, or $(b,c) \in I$;
\label{prop-item:invsB-planar}
\item $(a,b) \in I$ with $|a| \neq |b|$ implies that $(\bar b,  \bar a) \in I$.
\label{prop-item:invsB-neg}
\end{enumerate}
\end{prop}

In particular, Conditions~\eqref{prop-item:invsB-trans} and~\eqref{prop-item:invsB-planar} are the same as in type $A$, only Condition~\eqref{prop-item:invsB-neg} is added to take into account the symmetry of the signed permutation. In other words, an element of $W_n$ is only a permutation of size $2n$ (on the set $\Bsetnn$) with some extra symmetry conditions.

\begin{defi}\label{def:Standard}
Let $w$ be the word corresponding to an element of $W_n$. Let $I$ be a subset of $\Bsetnn$. We write $\bar{I}$ the set $\lbrace \bar a; a \in I \rbrace$. We say that $I$ is \emph{symmetric} if $I = \bar{I}$. On the contrary, we say that $I$ is {\em antisymmetric} if $I \cap \bar{I} = \varnothing$. Now
\begin{itemize}
\item for $I$ an antisymmetric subset of $\Bsetnn$\footnote{This definition is actually valid on any subset, but we only need in this article  it for antisymmetric subsets.}, then $\std(\rest{w}{I})$ is the classical standardization of the word $w$ restricted to $I$. The result is a permutation (of type $A$) of size $|I|$.
\item for $I$ a symmetric subset of $\Bsetnn$, we write $\Bstd(\rest{w}{I})$ the \emph{type $B$} standardization of the word $w$ restricted to $I$, \emph{i.e.}, this is the classical standardization except that we use the alphabet $\Bsetkk$ where $k = \frac{|I|}{2}$. The result is then an element of type $W_k$.
\end{itemize}
\end{defi}

\begin{ex}
\label{ex:restrition}
Let $w = 4 \bar{5} \bar{2} \bar{6} 3 1 | \bar{1} \bar{3} 6 2 5 \bar{4}$. Take $I = \lbrace \bar{5}, \bar{2}, 4, 6 \rbrace$. The subset $I$ is antisymmetric and we have $\rest{w}{I} = 4 \bar{5} \bar{2} 6$. The orders on the values is $\bar{5} < \bar{2} < 4 < 6$ which gives $\std(4 \bar{5} \bar{2} 6) = 3 1 2 4$, an element of $\Sym{4}$.

Now take $I = \lbrace \bar{6}, \bar{3}, \bar{2}, 2, 3, 6 \rbrace$ which is symmetric. We have $\rest{w}{I} = \bar{2} \bar{6} 3 | \bar{3} 6 2$ which gives $\Bstd(\rest{w}{I}) = \bar{1} \bar{3} 2 | \bar{2} 3 1$ an element of $B_3$. 
\end{ex}

\subsection{Recursive decomposition and proof of Theorem~\ref{thm:Main} for type $B$ permutations} For the remainding of this section, we identify the elements of~$W_n$ with their words as signed permutations.  The proof of Theorem~\ref{thm:Main} in the case of type $B$ is very similar to type $A$. The difficulty lies in finding a recursive decomposition of type $B$ permutations that satisfies the same properties as in type $A$.

\begin{defi}
Let $w$ be an element of $W_n$ with $n \geq 2$. The \emph{decreased word of}~$w$ is obtained by removing both the letters $n$ and $\bar{n}$ from $w$ such that $\dec{w}$ is now a word corresponding to an element of $W_{n-1}$.
\end{defi}

We use the same notation as in type $A$ as it is always clear from the context which decreasing is applied. For example, if $w = 4 \bar{5} \bar{2} \bar{6} 3 1 | \bar{1} \bar{3} 6 2 5 \bar{4}$, then $\dec{w} = 4 \bar{5} \bar{2} 3 1 | \bar{1} \bar{3} 2 5 \bar{4}$. We state the following result, which proof is similar to the one of type $A$ and is left to the reader.

\begin{lem}
\label{lem:B-dec}
Let $\lbrace u, v \rbrace$ be a bipartition of an element $w \in W_n$ with $n > 1$, then $\lbrace \dec{u}, \dec{v} \rbrace$ is a bipartition of $\dec{w}$.
\end{lem}

\begin{defi} Let $w\in W_n$. We say that {\em $n$ is in positive (resp. negative) position in $w$} if $n=w(i)$ (resp. $n=w(\bar i)$), for some $1\leq i\leq n$. 
\end{defi}

Write $t_1=\tau_0$ and $t_{i}=(\bar i\ i)=\sigma_{i-1}t_{i-1} \sigma_{i-1}$ for $2\leq i\leq n$.  The longest element of $W_n$ is well-known to be:
$$
w_\circ= n \dots 2\,1 |\bar 1  \bar 2\dots \bar n = t_1t_2\dots t_n.
$$
For any element $w\in W_n$, we have $ww_\circ (i)=\overline{w(i)}$ for all $i\in \Bsetnn$. Therefore:
$$
 n  \textrm{ is in {\em negative position} in }w \iff n \textrm{ is in {\em positive position} in }ww_\circ. 
$$
Let $\lbrace u, v \rbrace$ be a bipartition of an element $w \in W_n$ with $n > 1$. Assume that $n$ is in negative position in $w$. Then $(n,\bar n)\in\invs(w)$. Since the roles of $u$ and $v$ are symmetric, we have without loss of generality that  $n$ is in negative position in $u$ and is in positive position in $v$. Therefore $n$ is in positive position in $ww_\circ$, $uw_\circ$ and $v$. Moreover, by Proposition~\ref{prop:ReductionLongest}, $\{ww_\circ,v\}$ is a bipartition of $uw_\circ$ and Conjecture~\ref{conj:1} holds for the triplet $u,v,w$ if and only if it holds for $ww_\circ,v,uw_\circ$. The following proposition resume these facts.
\begin{prop}
\label{prop:B-pos}
Let $\lbrace u, v \rbrace$ be a bipartition of an element $w \in W_n$ with $n > 1$. Assume that $n$ is in negative position in $w$ and $u$. Then :
\begin{enumerate}
\item $n$ is in positive position in $uw_\circ$, $v$ and $ww_\circ$;
\item $\{ww_\circ,v\}$ is a bipartition of $uw_\circ$;
\item $d_R(u)+d_R(v)=d_R(w)$ if and only if $d_R(ww_\circ)+d_R(v)=d_R(uw_\circ)$.
\end{enumerate}
\end{prop}

The above proposition allows us to reduce the proof of Theorem~\ref{thm:Main}  for type $B$ to the set of elements of $W_n$ with $n$ in positive position. 

\begin{defi}
\label{def:Bdecomp}
Let $\lbrace u, v \rbrace$ be a bipartition of an element $w \in W_n$. Assume that~$n$ is in positive position in $w$. 

We define the \emph{right letters} of~$w$, $\R(w)$, the \emph{forgotten letters} of~$w$, $\F(w)$, and the \emph{left letters} of~$w$, $\L(w)$  as follows:

\begin{align*}
\R(w) &:= \lbrace a\mid (a,n) \in \invs(w) \rbrace \cup \lbrace n \rbrace; \\
\F(w) &:= \overline{\R(w)}; \\
\L(w) &:= \lbrace a \mid a \not\in \R(w) \cup \F(w) \rbrace.
\end{align*}
These letters are elements of $[\bar n,n]^*$. The {\em left} (resp. {\em right}) {\em words} of $w$, $u$, and $v$ are respectively:
\begin{align*}
w_\L &= \Bsrest{w}{\L(w)}, & w_\R &= \srest{w}{\R(w)} \\
u_\L &= \Bsrest{u}{\L(w)}, & u_\R &= \srest{u}{\R(w)} \\
v_\L &= \Bsrest{v}{\L(w)}, & v_\R &= \srest{v}{\R(w)}. 
\end{align*}
\end{defi}

\begin{remark}  Assume that~$n$ is in positive position in $w\in W$. Observe that   all letters of $\R(w)$ are in positive positions: so $\R(w)$ is an antisymmetric subset of $\Bsetnn$ (see Definition~\ref{def:Standard}). On the other hand, $\L(w)$ is defined as the complement of a symmetric set of letters, so it is itself a symmetric subset of $\Bsetnn$. 
\end{remark}

In the examples below,  whenever two letters belong to the same set $\L$, $\R$, or $\F$, we say that they are of the same \emph{color}. In the following, we often color left letters in red, right letters in blue and forgotten letters in black.

\begin{ex}[Continuation of Example~\ref{ex:TypeB2}]
\label{ex:typeB-pos1}
Consider $w = 4 \bar{5} \bar{2} \bar{6} 3 1 | \bar{1} \bar{3} 6 2 5 \bar{4}$. Observe that $n=6$ is in positive position in $w$. We claim that $u = 4 \bar{5} \bar{2} \bar{6} 3 \bar{1} | 1 \bar{3} 6 2 5 \bar{4}$ and $v = \bar{6} \bar{5} \bar{4} \bar{3} \bar{2} 1 | \bar{1} 2 3 4 5 6$ form a bipartition of $w$. Indeed, $v$ only contains the inversion $(\bar{1}, 1)$ and one can check that all other inversions of $w$ are also inversions of $u$. Conjecture~\ref{conj:1} holds in this case: 

\begin{align*}
\d(w) &= \d(4 \bar{5} \bar{2} \bar{6} 3 1 |. \bar{1}. \bar{3} 6. 2 5. \bar{4}), \\
\d(u) + \d(v) &= \d(4 \bar{5} \bar{2} \bar{6} 3 \bar{1} | 1 . \bar{3} 6 . 2 5 . \bar{4}) + d(\bar{6} \bar{5} \bar{4} \bar{3} \bar{2} 1 |. \bar{1} 2 3 4 5 6) = 3 + 1.
\end{align*}
Recall that in type $B$, $\d(w)$ only counts descents in nonnegative positions which we have marked with a dot for clarity.

\smallskip
\noindent
We obtain $\blue{\R(w)} = \lbrace \bar{4}, 2, 5, 6 \rbrace$, $\F(w) = \lbrace 4, \bar{2}, \bar{5}, \bar{6} \rbrace$, and $\red{\L(w)} = \lbrace \bar{3}, \bar{1}, 1, 3 \rbrace$. Therefore:

\begin{align*}
w &=   4 \bar{5} \bar{2} \bar{6} \red{ 3 1 | \bar{1} \bar{3}} \blue{ 6 2 5 \bar{4}} ,\\
u &= 4 \bar{5} \bar{2} \bar{6} \red{ 3 \bar{1} | 1 \bar{3}} \blue{ 6 2 5 \bar{4}}, \\
v &= \bar{6} \bar{5} \blue{\bar{4}} \red{\bar{3}} \bar{2} \red{1} | \red{\bar{1}} \blue{2} \red{3} 4 \blue{5 6}.
\end{align*}

\begin{align*}
w_\L &= \Bstd(\red{31 | \bar{1} \bar{3}}) = 2 1 | \bar{1} \bar{2}, & w_\R &= \std(\blue{6 2 5 \bar{4}}) = 4 2 3 1 ,\\
u_\L &= \Bstd(\red{3 \bar{1} | 1 \bar{3}}) = 2 \bar{1} | 1 \bar{2}, & u_\R &= \std(\blue{6 2 5 \bar{4}}) = 4 2 3 1, \\
v_\L &= \Bstd(\red{\bar{3} 1 | \bar{1} 3}) = \bar{2} 1 | \bar{1} 2, & v_\R &= \std(\blue{\bar{4} 2 5 6}) = 1 2 3 4. 
\end{align*}

\end{ex}

\begin{ex}
\label{ex:typeB-pos2}
Consider $w = 1 \bar{4} \bar{3} \bar{6} \bar{2} \bar{5} | 5 2 6 3 4 \bar{1}$. Observe that $n=6$ is in positive position in $w$. The elements $u = \bar{4} \bar{3} \bar{6} \bar{2} \bar{5} \bar{1} | 1 5 2 6 3 4$ and $v = 1 \bar{6} \bar{5} \bar{4} \bar{3} \bar{2} | 2 3 4 5 6 \bar{1}$ form a bipartition of $w$. We obtain $\blue{\R(w)} = \lbrace \bar{1}, 3, 4, 6 \rbrace$, $\F(w) = \lbrace 1, \bar{3}, \bar{4}, \bar{6} \rbrace$, and $\red{\L(w)} = \lbrace \bar{5}, \bar{2}, 2, 5 \rbrace$. Hence: 

\begin{align*}
w &= 1 \bar{4} \bar{3} \bar{6}\red{ \bar{2} \bar{5} | 5 2} \blue{ 6 3 4 \bar{1}} ,\\
u &= \bar{4} \bar{3} \bar{6} \red{\bar{2} \bar{5}} \blue{\bar{1}} | 1 \red{5 2} \blue{ 6 3 4} ,\\
v &= 1 \bar{6} \red{\bar{5}} \bar{4} \bar{3} \red{\bar{2}} | \red{2} \blue{3 4} \red{5} \blue{ 6 \bar{1}}.
\end{align*}

\begin{align*}
w_\L &= \Bstd(\red{\bar{2} \bar{5} | 5 2}) = \bar{1} \bar{2} | 2 1, & w_\R &= \std(\blue{6 3 4 \bar{1}}) = 4 2 3 1, \\
u_\L &= \Bstd(\red{\bar{2} \bar{5} | 5 2}) = \bar{1} \bar{2} | 2 1, & u_\R &= \std(\blue{\bar{1} 6 3 4}) = 1 4 2 3 ,\\
v_\L &= \Bstd(\red{\bar{5} \bar{2} | 2 5}) = \bar{2} \bar{1} | 1 2, & v_\R &= \std(\blue{3 4 6 \bar{1}}) = 2 3 4 1. 
\end{align*}

\end{ex}

\begin{ex}
\label{ex:typeB-pos3}
Consider $w = 2 \bar{1} \bar{3} \bar{4} 5 \bar{6} | 6 \bar{5} 4 3 1 \bar{2}$,. Observe that $n=6$ is in positive position in $w$. The elements  $u = 2 \bar{1} \bar{4} \bar{3} 5 \bar{6} | 6 \bar{5} 3 4 1 \bar{2}$ and $v = \bar{6} \bar{5} \bar{3} \bar{4} \bar{2} \bar{1} | 1 2 4 3 5 6$ form a bipartition of $w$.  As $n$ is in position $1$, $\blue{\R(w)}$ contains all letters in positive positions while $\F(w)$ contains all letters in negative positions and $\red{\L(w)}$ is empty. We get

\begin{align*}
w &= 2 \bar{1} \bar{3} \bar{4} 5 \bar{6} | \blue{6 \bar{5} 4 3 1 \bar{2}}, \\
u &= 2 \bar{1} \bar{4} \bar{3} 5 \bar{6} | \blue{6 \bar{5} 3 4 1 \bar{2}} ,\\
v &= \bar{6} \blue{\bar{5}} \bar{3} \bar{4} \blue{\bar{2}} \bar{1} |  \blue{1} 2 \blue{ 4 3 } 5 \blue{ 6}
\end{align*}

\begin{align*}
w_\R &= \std(\blue{6 \bar{5} 4 3 1 \bar{2}}) = 6 1 5 4 3 2, \\
u_\R &= \std(\blue{6 \bar{5} 3 4 1 \bar{2}}) = 6 1 4 5 3 2 ,\\
v_\R &= \std(\blue{\bar{5} \bar{2} 1 4 3 6}) = 1 2 3 5 4 6. 
\end{align*}

\end{ex}

Similarly to the the case of type $A$, we ``cut'' the word $w$ using the position of $n$ (and $\bar{n}$) into three parts for which the forgotten letters are always the symmetric of the right letters by definition.

\begin{lem}
\label{lem:b-cut}
Let $w\in W_n$ and assume that~$n$ is in positive position in $w$. Then $w$ is of the form:

\begin{equation*}
\begin{array}{ccc}
\dots \bar{n}   & \red{\dots | \dots} & \blue{n \dots} \\
\F(w)           & \red{\L(w)}         & \blue{\R(w)},
\end{array}
\end{equation*}
\emph{i.e.}, it contains first all the forgotten letters ending with $\bar{n}$ all in negative positions, then the left letters in negative then positive positions, then all the right letters starting with $n$. Moreover $\overline{\F(w)}=\R(w)$ and $\L(w)$ is a symmetric subset of $\Bsetnn$.
\end{lem}

\begin{proof}
This is direct from the definition. Since $n$ is in position $k > 0$ and $\bar{n}$ is in position $\bar k$. The right letters are all letters $a < n$ such that $(a,n)$ is an inversion of $w$: they are all the letters placed to the right of $n$, in positive positions by definition. The forgotten letters are exactly the opposite of the right letters which by symmetry of $w$ are placed in all negative positions before $\bar k$. In the middle of the word, we obtain the set $\L(w)$ of left letter, which is a symmetric subset of $\Bsetnn$.
\end{proof}

The following proposition is the equivalent in the case of type $B$ of Propositions~\ref{prop:decomposition} in type $A$.

\begin{prop}
\label{prop:B-decomposition}
Let $w\in W_n$ such that $n$ is in positive position in $w$. Let $\lbrace u, v \rbrace$ be a bipartition of an element $w$. Then the left words $\lbrace u_\L, v_\L \rbrace$ form a bipartition of $w_\L$ and the right words $\lbrace u_\R, v_\R \rbrace$ form a bipartition of $w_\R$ and

\begin{align}
\label{eq:B-deswdecomp}
\d(w) &= \d(w_\L) + \d(w_\R), \\
\label{eq:B-desudecomp}
\d(u) &= \d(u_\L) + \d(u_\R), \\
\label{eq:B-desvdecomp}
\d(v) &= \d(v_\L) + \d(v_\R).
\end{align}
\end{prop}

Note that $w_\L$, $u_\L$, $v_\L$ and $w_\R$, $u_\R$, $v_\R$ belong to different types of Coxeter groups:  the left words are of type $B$ and the right words are of type $A$.

\begin{ex}[Continuation of Example~\ref{ex:typeB-pos1}]
Here, $w = 4 \bar{5} \bar{2} \bar{6} \red{ 3 1 | \bar{1} \bar{3}} \blue{ 6 2 5 \bar{4}}$ with $u = 4 \bar{5} \bar{2} \bar{6} \red{ 3 \bar{1} | 1 \bar{3}} \blue{ 6 2 5 \bar{4}}$ and $v = \bar{6} \bar{5} \blue{\bar{4}} \red{\bar{3}} \bar{2} \red{1} | \red{\bar{1}} \blue{2} \red{3} 4 \blue{5 6}$ forming a bipartition of $w$. We obtain that $u_\L = 2 \bar{1} | 1 \bar{2}$ and $v_\L = \bar{2} 1 | \bar{1} 2$ form a bipartition of $w_L = 2 1 | \bar{1} \bar{2}$ in~$W_2$, while $u_\R = 4 2 3 1$ and $v_\R = 1 2 3 4$ form a (trivial) bipartition of $w_\R = u_\R$ in~$\mathcal S_3$. Moreover, we have

\begin{align*}
4 &=& \d(w) &=& \d(4 \bar{5} \bar{2} \bar{6} \red{ 3 1 |. \bar{1}. \bar{3}} \blue{ 6 . 2 5 . \bar{4}}) &=& \d(2 1 |. \bar{1}. \bar{2}) + \d(4. 2 3 . 1) &=& 2 + 2 \\
3 &=& \d(u) &=& \d(4 \bar{5} \bar{2} \bar{6} \red{ 3 \bar{1} | 1 . \bar{3}} \blue{ 6 . 2 5 . \bar{4}}) &=& \d(2 \bar{1} | 1 . \bar{2}) + \d(4. 2 3 . 1) &=& 1 + 2 \\
1 &=& \d(v) &=& \d(\bar{6} \bar{5} \blue{\bar{4}} \red{\bar{3}} \bar{2} \red{1 | . \bar{1}} \blue{2} \red{3} 4 \blue{5 6}) &=& \d(\bar{2} 1 |. \bar{1} 2) + \d(1 2 3 4) &=& 1 + 0
\end{align*}
\end{ex}

We prove Proposition~\ref{prop:B-decomposition} in the next section. We first show how Conjecture~\ref{conj:1} is proved in the case of type $B$ using Propositions~\ref{prop:B-decomposition} and~Theorem~\ref{thm:Main} in the case of type $A$.

\begin{proof}[Proof of Conjecture~\ref{conj:1} in type $B$] As in type $A$, the proof is  by induction on $n\in \mathbb N^*$. The result is true for $n=1$ since  $W_1$ is of type $B_1=A_1$. Now, let $\lbrace u, v \rbrace$ be a bipartition of an element $w \in W_n$ with $n \geq 2$. We suppose that the statement of Conjecture~\ref{conj:1} is satisfied for all $w\in W_k$ with $1 \leq k < n$ such that $n$ is in positive position in $w$. 

Suppose first that $|\L(w)|=k$ with $k > 0$. By Proposition~\ref{prop:B-decomposition}, $\lbrace u_\L, v_\L \rbrace$ is a bipartition of $w_\L$, and~$\lbrace u_\R, v_\R \rbrace$ is a bipartition of $w_\R$,  $w_\L$ is an element of $W_k$, with $k < n$, and $w_\R$ is an element of $A_{n-1-k}$. Therefore $\lbrace u_\L, v_\L \rbrace$ satisfies Conjecture~\ref{conj:1} by induction while $\lbrace u_\R, v_\R \rbrace$ satisfies Conjecture~\ref{conj:1} using Theorem~\ref{thm:Main} in the case of type $A$. Hence: 
$$
\d(w_\L) = \d(u_\L) + \d(v_\L)\textrm{ and }\d(w_\R) = \d(u_\R) + \d(v_\R).
$$ 
Therefore $\d(w) = \d(u) + \d(v)$, by Eq.~\eqref{eq:B-deswdecomp}, Eq.~\eqref{eq:B-desudecomp}, and~Eq.~\eqref{eq:B-desvdecomp} of Proposition~\ref{prop:B-decomposition}.

Now suppose that $\L(w)=\emptyset$. Therefore~$n$ is positive position~$1$ (as in Example~\ref{ex:typeB-pos3}). Hence $w_\R\not =w$ is a standardization of the letters in positive positions in $w$ and an element of $\mathcal S_{n-1}$ such that $\d(w_\R) = \d(w)$. By Proposition~\ref{prop:B-decomposition}, $\lbrace u_\R, v_\R \rbrace$ form a bipartition of~$w_\R$ with $\d(u_R) = \d(u)$, $\d(v_\R) = \d(w)$. Applying Theorem~\ref{thm:Main} in  the case of type type $A$, we obtain $\d(w) = \d(u) + \d(v)$. 
\end{proof}

\subsection{Proof of Proposition~\ref{prop:B-decomposition}}

The proof is  similar to the one valid in type $A$, we start by proving the type $B$ analog of Lemma~\ref{lem:right-left-inversions}.

\begin{lem}
\label{lem:B-right-left-inversions}
Let $w\in W_n$ such that $n$ is in positive position in $w$. Let $\lbrace u, v \rbrace$ be a bipartition of an element $w$. Then there is no inversion $(a,c)$ of $u$ or $v$ such that $a \in \F(w) \cup \L(w)$ and $c \in \R(w)$, nor such that $a\in \F(w)$ and $c \in \L(w) \cup \R(w)$. In particular, $n$ is placed after all letters of $\F(w)$ and $\L(w)$ in both $u$ and $v$.
\end{lem} 

\begin{proof}
This is a direct consequence of Lemma~\ref{lem:b-cut} describing how the left, right, and forgotten letters are placed in $w$ and of the fact that all inversions of $u$ and $v$ are inversions of $w$.
\end{proof}

\begin{prop}
\label{prop:B-decents}
Let $w\in W_n$ such that $n$ is in positive position in $w$. Let $\lbrace u, v \rbrace$ be a bipartition of an element $w$. If there is a descent in the word $u$ (or in the word $v$) between two letters $c$ and $a$, then either $a,c\in\R(w)$, or $a,c\in\L(w)$, or $a,c\in\F(w)$. In other words, a descent can only happen between two letters of similar colors.
\end{prop} 

This statement is true for all descents, not only for the right descents, i.e.,  in nonnegative positions.

\begin{ex}[Continuation of example~\ref{ex:typeB-pos1}]
Recall that $w = 4 \bar{5} \bar{2} \bar{6} \red{ 3 1 | \bar{1} \bar{3}} \blue{ 6 2 5 \bar{4}}$ with $u = 4 \bar{5} \bar{2} \bar{6} \red{ 3 \bar{1} | 1 \bar{3}} \blue{ 6 2 5 \bar{4}}$ and $v = \bar{6} \bar{5} \blue{\bar{4}} \red{\bar{3}} \bar{2} \red{1} | \red{\bar{1}} \blue{2} \red{3} 4 \blue{5 6}$ as  bipartition. Now,  all descents in $u$ and $v$ are marked with a dot: observe that there are only descents  between letters of the same color:
$
u = 4 .\bar{5}\bar{2}. \bar{6} \red{3 . \bar{1} } \red{| 1 . \bar{3}} \blue{6 . 2 5 . \bar{4}} \textrm{ and }
v = \bar{6} \bar{5} \blue{\bar{4}} \red{\bar{3}} \bar{2} \red{1 } \red{|. \bar{1}} \blue{2} \red{3} 4 \blue{5 6}.
$

\end{ex}

Before proving Proposition~\ref{prop:B-decents}, we make an observation on the proof Proposition~\ref{prop:decents} in type $A$: we notice that the reasoning does not depend of the type and can actually be applied in certain cases in type $B$ as stated in the following lemma.

\begin{lem}[Corollary of Proposition~\ref{prop:decents}]
\label{lem:B-A-descent-cases}
Let $w\in W_n$ such that $n$ is in positive position in $w$. Let $\lbrace u, v \rbrace$ be a bipartition of an element $w$. Then if there is a descent in~$u$ (or in~$v$) between two letters~$c$ and~$a$, then either $a,c\in\R(w)$ or $a,c\in\L(w) \cup \F(w)$.
\end{lem}

\begin{proof}
Consider the form of the word $w$ in Lemma~\ref{lem:b-cut} and observe that the definition of the left letters in type $A$ (everything before the letter $n$ in $w$) corresponds to $\L(w) \cup \F(w)$. Similarly, the right letters of type $A$ correspond to $\R(w)$ in type~$B$. The exact same line of reasoning used to prove Proposition~\ref{prop:decents} applies here with Lemma~\ref{lem:B-right-left-inversions} and Proposition~\ref{prop:invsB}.
\end{proof}

\begin{proof}[Proof of Proposition~\ref{prop:B-decents}]
Suppose that $u$ is of the form $\dots ca \dots$ with $a < c$. Note that we do not specify if the descent is in positive, negative or $0$ position as this does not impact the proof. In particular, we could have $a = \bar c$ if the descent is in position~$0$.  Lemma~\ref{lem:B-A-descent-cases} states that either $a,c \in \R(w)$ or $a,c \in \L(w) \cup \F(w)$. By Lemma~\ref{lem:B-right-left-inversions}, we cannot have $c \in \L(w)$ and $a \in \F(w)$. The only case to consider is then $c \in \F(w)$ and $a \in \L(w)$ which we suppose by contradiction. As $\L(w)$ is symmetric, we also have $\bar a \in \L(w)$ and in particular $\bar a \neq c$ and the descent is not in position $0$. By symmetry, this implies that there is a descent between $\bar a \in \L(w)$ and $\bar c \in \R(w)$ (because $\F(w) = \overline{\R(w)}$) which is impossible accordingly to Lemma~\ref{lem:B-A-descent-cases}.
\end{proof}

\begin{prop}
\label{prop:B-no-extra-desc}
Let $w\in W_n$ such that $n$ is in positive position in $w$. Let $\lbrace u, v \rbrace$ be a bipartition of an element $w$.  Assume that $u$ (or $v$) is of the form $\dots b \dots c \dots$ with $b$ and $c$  both in the same set $I$ such that: (1)  $I=\L(w)$, $I=\R(w)$ or $I=\F(w)$; (2) the subword between $b$ and $c$ is non-empty  all its letters are in $\Bsetnn\setminus I$. Then $b < c$.
\end{prop}

\begin{proof}
We suppose by contradiction that $u = \dots c u' b \dots$ with $b < c$ and $b,c \in I$ with  $I=\L(w)$, $I=\R(w)$ or $I=\F(w)$ such that $u'$ is not empty and with letters not in $I$. We proceed to a case by case analysis and show that we always reach a contradiction

\smallskip
\noindent\textbf{Case $b,c \in \L(w)$.}  The word $u'$ is non-empty by hypothesis, we call $a$ the last letter of $u'$ so $u$ is of the form $\dots c \dots a b \dots$. As $a \notin \L(w)$, there is no descent between $a$ and $b$ by Proposition~\ref{prop:B-decents} so $a < b$. In particular, $(a,c) \in \invs(u)$. As $c \in \L(w)$, by Lemma~\ref{lem:B-right-left-inversions}, we cannot have $a \in \F(w)$, so $a \in \R(w)$. Besides, as $(a,c) \in \invs(u)$, we have $(a,c) \not\in \invs(v)$. By Lemma~\ref{lem:B-right-left-inversions}, we also know that $n$ is placed after all letters of $\L(w)$ in both $u$ and $v$, this gives

\begin{align*}
u &= \dots c \dots ab \dots n \dots \\
v &= \dots a \dots c \dots n \dots.
\end{align*}
We see that $(a,n) \not\in \invs(u)$ and $(a,n) \not\in \invs(v)$. This contradicts the fact that $(a,n)\in\invs(w)$ since $a$ is a right letter of $w$.

\smallskip
\noindent\textbf{Case $b,c \in \R(w)$.} We call $a$ the last letter of $u'$ so that $u = \dots c \dots ab \dots$. As before, because of Proposition~\ref{prop:B-decents}, $a<b$. We have $a \in \L(w) \cup \F(w)$ and $(a,c) \in \invs(u)$: this contradicts Lemma~\ref{lem:B-right-left-inversions}.

\smallskip
\noindent\textbf{Case  $b,c \in \F(w)$.} Obtained by symmetry of the case $b,c \in \R(w)$. 
\end{proof}

We are finally ready to prove Proposition~\ref{prop:B-decomposition}, hence conluding the proof of Theorem~\ref{thm:Main} in the case of type $B$.

\begin{proof}[Proof of Proposition~\ref{prop:B-decomposition}]
Let $w\in W_n$ such that $n$ is in positive position in $w$. Let $\lbrace u, v \rbrace$ be a bipartition of an element $w$. By a similar argument as in the case of type~$A$, we obtain that $\lbrace u_\L, v_\L \rbrace$ form a bipartition of $w_\L$ and $\lbrace u_\R, v_\R \rbrace$ form a bipartition of $w_\R$, by simply restricting the sets of inversions to the left word (resp. right word) letters and using Lemma~\ref{lem:B-invsT}.

We first show Equality~\eqref{eq:B-deswdecomp}. We use Lemma~\ref{lem:b-cut} and suppose that $n$ is in position $k > 0$. There is a descent at position $ 0 \leq i < k - 1$ in $w$ if and only if there is a descent at position $i$ in $w_\L$ as $w_\L$ is the type $B$ standardization of the middle factor of $w$. Similarly, there is a descent at position $i > k$ in $w$ if and only if there is a descent at position $i-k+1$ in $w_\R$. Besides, there is no descent in $w$ at position $k-1$. We find $\d(w) = \d(w_\L) + \d(w_\R)$ (as $\d(w)$ and $\d(w_\L)$ only count descents in nonnegative positions).

We now prove Equation~\eqref{eq:B-desudecomp} (which is equivalent to Equation~\eqref{eq:B-desvdecomp}). Similarly as in type~$A$, Proposition~\ref{prop:B-decents} implies:
\begin{equation*}
\d(u) \leq \d(u_\L) + \d(u_\R).
\end{equation*} 
Indeed, suppose that we have a descent $ca$ in $u$ in positive position. Either $a,c \in \L(w)$ and the descent is kept in $u_\L$, still in positive position. Or $a,c \in \R(w)$ and the descent is kept in $u_\R$. Finally, if $a,c \in \F(w)$, then $u$ also has a descent in negative position between $\bar a$ and $\bar c$ in $\R(w)$ which was not counted in $\d(u)$ but is kept in $u_\R$ and counted in $\d(u_\R)$. 

We now claim that Proposition~\ref{prop:B-no-extra-desc} implies:
\begin{equation*}
\d(u_\L) + \d(u_\R) \leq \d(u).
\end{equation*} 
Suppose that there is a descent $ca$ in $u_\L$ in nonnegative position. This means that $u$ is of the form $\dots c' \dots a' \dots$ where $a' < c'$ are letters of $\L(w)$ which are standardized into $a$ and $c$ in $u_\L$. As they are consecutive in $u_\L$, this means that all letters between $c'$ and $a'$ are not in $\L(w)$. Then Proposition~\ref{prop:B-no-extra-desc} forces the word between $c'$ and $a'$ to be empty. Therefore $c'a'$ is a descent of $u$. 
Moreover, since $\L(w)$ is a symmetric subset of $\Bsetnn$, the $B$ standardization cannot change a negative position into a positive one and the descent $c'a'$ in $u$ is then in nonnegative position; therefore $c'a'$ is counted in $\d(u)$. 

Now suppose that there is a descent $ca$ in $u_\R$ (so in positive position). With similar arguments, using Proposition~\ref{prop:B-no-extra-desc}, we get that there is a descent $c'a'$ in $u$ with $a',c' \in \R(w)$. If $c'a'$ is in nonnegative position,  $c'a'$ is then counted in $\d(u)$. If $c'a'$ is in negative position, then $\overline{a'}~\overline{c'}$ is a descent in positive position with letters in $\F(w)$ which is counted in $\d(u)$ and was not counted before. This implies the desires inequality and concludes the proof.
\end{proof}

\subsection{A remark on the recursive decomposition}

Thanks to Proposition~\ref{prop:B-pos}, we only need to consider the case of $n$ in a positive position in $w\in W$. However, we could have defined a recursive decomposition also in the case where $n$ is in a negative position in $w$ and all the above results would have still been valid; however the proofs would have been longer and would have required to prove a type $B$ analog of Proposition~\ref{prop:decreasing}, which holds in that case too.

\begin{defi}[The letter $n$ is in a negative position] Let $w$ in $W$ such that $n$ is in a negative position in $w$. We define the \emph{left letters} of~$w$, $\L(w)$, the \emph{forgotten letters} of~$w$, $\F(w)$, and the \emph{right letters} of~$w$, $\R(w)$,  as

\begin{align*}
\L(w) &:= \lbrace a < n \mid (a,n) \not\in \invs(w) \rbrace; \\
\F(w) &:= \overline{\L(w)}. \\
\R(w) &:= \lbrace a \mid a \not\in \L(w) \cup \F(w) \rbrace;
\end{align*}

The {\em left} (resp. {\em right}) {\em words} of $w$, $u$, and $v$ are respectively:
\begin{align*}
w_\L &= \srest{w}{\L(w)}, & w_\R &= \Bsrest{w}{\R(w)} \\
u_\L &= \srest{u}{\L(w)}, & u_\R &= \Bsrest{u}{\R(w)} \\
v_\L &= \srest{v}{\L(w)}, & v_\R &= \Bsrest{v}{\R(w)}. 
\end{align*}
\end{defi}

\begin{ex}
\label{ex:typeB-neg1}
Consider $w = 1 \bar{3} 6 2 \bar{4} \bar{5} | 5 4 \bar{2} \bar{6} 3 \bar{1}$, then $u = 1 \bar{3} 6 2 \bar{5} \bar{4} | 4 5 \bar{2} \bar{6} 3 \bar{1}$ and $v = \bar{6} \bar{4} \bar{5} \bar{3} \bar{2} \bar{1} | 1 2 3 5 4 6$ form a bipartition of $w$.  We obtain $\L(w) = \lbrace \bar{3}, 1 \rbrace$, $\F(w) = \lbrace \bar{1}, 3 \rbrace$, and $\R(w) = \lbrace \bar{6}, \bar{5}, \bar{4}, \bar{2}, 2, 4, 5, 6 \rbrace$. This gives

\begin{align*}
w &= \red{1 \bar{3}} \blue{ 6 2 \bar{4} \bar{5} | 5 4 \bar{2} \bar{6}} 3 \bar{1}, \\
u &= \red{1 \bar{3}} \blue{ 6 2 \bar{5} \bar{4} | 4 5 \bar{2} \bar{6}} 3 \bar{1}, \\
v &= \blue{\bar{6} \bar{4} \bar{5}} \red{\bar{3}} \blue{\bar{2}} \bar{1} | \red{1} \blue{2} 3 \blue{5 4 6}.
\end{align*}

\begin{align*}
w_\L &= \std(\red{1 \bar{3}}) = 2 1, & w_\R &= \Bstd(\blue{6 2 \bar{4} \bar{5} | 5 4 \bar{2} \bar{6}}) = 4 1 \bar{2} \bar{3} | 3 2 \bar{1} \bar{4}, \\
u_\L &= \std(\red{1 \bar{3}}) = 2 1, & u_\R &= \Bstd(\blue{6 2 \bar{5} \bar{4} | 4 5 \bar{2} \bar{6}}) = 4 1 \bar{3} \bar{2} | 2 3 \bar{1} \bar{4}, \\
v_\L &= \std(\red{\bar{3} 1}) = 1 2, & v_\R &= \Bstd(\blue{\bar{6} \bar{4} \bar{5} \bar{2} | 2 5 4 6}) = \bar{4} \bar{2} \bar{3} \bar{1} | 1 3 2 4. 
\end{align*}

\end{ex}

\begin{ex}
\label{ex:typeB-neg2}
Consider $w = 6 5 \bar{3} 4 \bar{1} 2 | \bar{2} 1 \bar{4} 3 \bar{5} \bar{6}$, then $u = \bar{5} 6 \bar{3} 4 \bar{2} \bar{1} | 1 2 \bar{4} 3 \bar{6} 5$, and $v = 5 \bar{6} \bar{4} \bar{3} \bar{1} 2 | \bar{2} 1 3 4 6 \bar{5}$ form a bipartition of $w$.  Since~$n$ is in first position ($\bar{n}$ in last), there are no left letters, nor forgotten letters and $w_\L = w$, $u_\L =u$, $v_\L = v$.
\end{ex}

\section{On  enumeration of bipartitions and partition-irreducible elements}\label{ss:Computations}

We present below some of the enumerative datas we obtained with Sagemath using the characterization of bipartitions of elements of $W$ given in Proposition~\ref{prop:rectangled}. For each of these examples, Conjecture~\ref{conj:1} (and therefore Conjecture~\ref{conj:2}) is satisfied for all $w\in W$ up to length specified in Table~\ref{tab:ex1} and Table~\ref{tab:ex2}.

\smallskip

In the tables below, we consider the following subsets of $W$:
\def\bip{\textnormal{Bip}}
\def\pirr{\textnormal{PIrr}}
$$
\bip(W,S)=\{w\in W\mid w\textrm{ admits at least one proper bipartition}\}
$$
and
$$
 \pirr(W,S)=\{w\in W\mid w\textrm{ is partition-irreducible}\}.
$$
We associate to these sets the following generating functions (where $k\in\mathbb N$):
$$
\bip(q,k)=\sum_{\substack{w\in \bip(W,S)\\ \ell(w)\leq k}} q^\ell(w)\quad\textrm{and}\quad\bip(q)=\sum_{w\in \bip(W,S)} q^\ell(w);
$$
$$
\pirr(q,k)=\sum_{\substack{w\in \pirr(W,S)\\ \ell(w)\leq k}} q^\ell(w)\quad\textrm{and}\quad \pirr(q)=\sum_{w\in \pirr(W,S)} q^\ell(w).
$$
For a finite Coxeter group $W$, we give the generating functions $\bip(q)$ and $ \pirr(q)$, that is, for $k=\ell(w_\circ)$. 

\begin{remarks}
\begin{enumerate}
\item In type A and B/C, some enumerative results are given in \cite[\S6]{DeDiRo17}. It would be interesting to study if those results could be easily generalized to any Coxeter system, or at least for any finite Coxeter system. In particular, Catalan numbers appear in  \cite[Lemma 6.1]{DeDiRo17}. 
\item Computations under sagemath using Pad\'e's approximant suggests that none of the generating functions in Table~\ref{tab:ex1} and Table~\ref{tab:ex2} are rationnal.
\item It would be very interesting to have a better characterization of partition-irreducible elements. 
\end{enumerate}
\end{remarks}

\begin{figure}
\def\arraystretch{1.25}
\begin{center} 
\begin{tabular}{|c|c|c|}
\hline
Type of   &   $\bip(q,k)$   & $k\in\mathbb N$ such that  \\
of $(W,S)$  & $ $   &$\ell(w)\leq k$  \\
  \hline \hline
  $A_2$& $q^3$&$3=\ell(w_\circ)$\\
  \hline
  $A_3$& $q^2 + 2q^3 + 4q^4 + 3q^5 + q^6$&$6=\ell(w_\circ)$\\
  \hline
  $A_4$& $3q^2 + 7q^3 + 10q^4 + 16q^5 + 16q^6 $&$10=\ell(w_\circ)$\\
  &$+ 15q^7 + 9q^8 + 4q^9 + q^{10}$& \\
    \hline
  $A_5$& $6q^2 + 17q^3 + 30q^4 + 43q^5 + 62q^6 + 78q^7 + 74q^8   $&$15=\ell(w_\circ)$\\
  &$+ 79q^9+ 67q^{10} + 49q^{11} + 29q^{12} + 14q^{13} + 5q^{14} + q^{15}$& \\
 \hline
  $B_2$& $q^4 $&$4=\ell(w_\circ)$\\ 
   \hline
  $B_3$& $q^2 + q^3 + 3q^4 + 2q^5 + 4q^6 + 4q^7 + 3q^8 + q^9 $&$9=\ell(w_\circ)$\\ 
   \hline
  $B_4$& $3q^2 + 6q^3 + 10q^4 + 12q^5 + 19q^6 + 17q^7  $&$16=\ell(w_\circ)$\\ 
    &$+ 21q^8 + 19q^9 + 22q^{10} + 23q^{11} + 19q^{12} $& \\
        &$+ 16q^{13} + 9q^{14} + 4q^{15} + q^{16}$& \\
     \hline
  $D_4$& $3q^2 + 4q^3 + 12q^4 + 15q^5 + 15q^6 + 15q^7 $&$12=\ell(w_\circ)$\\
       &$+ 21q^8 + 15q^9 + 9q^{10} + 4q^{11} + q^{12} $& \\
    \hline
  $D_5$& $ 6q^2 + 14q^3 + 26q^4 + 44q^5 + 65q^6 + 78q^7 + 99q^8  $&$20=\ell(w_\circ)$\\
       &$+ 114q^9 + 103q^{10} + 115q^{11} + 122q^{12} + 101q^{13}  + 100q^{14}   $& \\
       &$+ 80q^{15} + 54q^{16} + 30q^{17} + 14q^{18} + 5q^{19} + q^{20} $& \\
 \hline
  $F_4$& $ 3q^2 + 6q^3 + 7q^4 + 12q^5 + 10q^6 + 20q^7 + 24q^8 + 26q^9    $&$24=\ell(w_\circ)$\\
       &$+ 26q^{10} + 22q^{11} + 28q^{12} + 22q^{13} + 26q^{14} + 26q^{15} + 28q^{16}   $& \\
       &$+ 24q^{17}+ 26q^{18} + 20q^{19} + 21q^{20} + 16q^{21} + 9q^{22} + 4q^{23} + q^{24}  $& \\
 \hline
  $H_3$& $q^2 + q^3 + 2q^4 + q^5 + 3q^6 + 2q^7 + 4q^8   $&$15=\ell(w_\circ)$\\ 
    &$+ 2q^9 + 3q^{10} + 2q^{11} + 4q^{12} + 4q^{13} + 3q^{14} + q^{15}$& \\
\hline
 \hline
  $\tilde A_2$&$3q^3 + 6q^5 + 6q^7 + 6q^9 + 6q^{11}  $&$12$\\
   \hline
  $\tilde B_2$& $ 2q^2 + 4q^3 + 16q^4 + 20q^5 + 12q^6 + 32q^7   $&$12$\\ 
    &$+ 44q^8 + 24q^9 + 44q^{10} + 56q^{11} + 36q^{12}  $& \\
    \hline
  $\tilde G_2$& $q^2 + 2q^4 + 4q^5 + 2q^6 + 2q^7    $&$12$\\ 
    &$+ 2q^8 + 4q^9 + 4q^{10} + 2q^{12} $& \\

    \hline \hline
  \begin{tikzpicture}
	[scale=1,
	 sommet/.style={inner sep=2pt,circle,draw=black,fill=blue,thick,anchor=base},
	 rotate=0,
	 baseline = 0]
 \tikzstyle{every node}=[font=\small]
\coordinate (ancre) at (0,-0.0);
\node  at ($(ancre)+(0,0.6)$) {};
\node[sommet] (a2) at ($(ancre)+(0.25,0.5)$) {};
\node[sommet] (a3) at ($(ancre)+(0.5,0)$) {} edge[thick] node[auto,swap,right] {$ $} (a2) ;
\node[sommet] (a4) at (ancre) {} edge[thick] node[auto,swap,below] {$ $} (a3) edge[thick] node[auto,swap,left] {$4$} (a2);
\end{tikzpicture}
&$ 2q^3 + q^4 + 2q^5 + 4q^6 + 6q^8+ 4q^{10} + 2q^{11} + 2q^{12}$&$12$\\
  \hline
   \begin{tikzpicture}
	[scale=1,
	 sommet/.style={inner sep=2pt,circle,draw=black,fill=blue,thick,anchor=base},
	 rotate=0,
	 baseline = 0]
 \tikzstyle{every node}=[font=\small]
\coordinate (ancre) at (0,-0.0);
\node  at ($(ancre)+(0,0.6)$) {};
\node[sommet] (a1) at ($(ancre)+(0,0.5)$) {};
\node[sommet] (a2) at ($(ancre)+(0.5,0.5)$) {} edge[thick] node[auto,swap,right] {} (a1) edge[thick] node[auto,swap,left] {} (a3);
\node[sommet] (a3) at ($(ancre)+(0.5,0)$) {} edge[thick] node[auto,swap,right] {} (a1) ;
\node[sommet] (a4) at (ancre) {} edge[thick] node[auto,swap,below] {} (a3) ;
\end{tikzpicture}
&$ 2q^2 + 6q^3 + 9q^4 + 16q^5 + 16q^6 + 22q^7  $&$12$\\
& $+ 27q^8 + 28q^9 + 32q^{10} + 40q^{11} + 50q^{12}$& \\
  \hline
\end{tabular}
  \captionof{table}{Generating functions of elements admitting a bipartition up to length $k$}\label{tab:ex1}
\end{center}
\end{figure}

\begin{figure}
\def\arraystretch{1.25}
\begin{center} 
\begin{tabular}{|c|c|c|}
\hline
Type of   &   $\pirr(q,k)$   & $k\in\mathbb N$ such that  \\
of $(W,S)$  & $ $   &$\ell(w)\leq k$  \\
  \hline \hline
  $A_2$& $1+2q+2q^2$&$3=\ell(w_\circ)$\\
  \hline
  $A_3$& $1 + 3q + 4q^2 + 4q^3 + q^4$&$6=\ell(w_\circ)$\\
  \hline
  $A_4$& $1 + 4q + 6q^2 + 8q^3 + 10q^4 + 6q^5 + 4q^6$&$10=\ell(w_\circ)$\\
    \hline
  $A_5$& $1 + 5q + 8q^2 + 12q^3 + 19q^4 + 28q^5 $&$15=\ell(w_\circ)$\\
  &$+ 28q^6 + 23q^7 + 27q^8 + 11q^9 + 4q^{10}$& \\
 \hline
  $B_2$& $1 + 2q + 2q^2 + 2q^3 $&$4=\ell(w_\circ)$\\ 
   \hline
  $B_3$& $1 + 3q + 4q^2 + 6q^3 + 5q^4 + 6q^5 + 3q^6 + q^7 $&$9=\ell(w_\circ)$\\ 
   \hline
  $B_4$& $1 + 4q + 6q^2 + 10q^3 + 14q^4 + 20q^5 + 20q^6  $&$16=\ell(w_\circ)$\\ 
    &$+ 27q^7 + 25q^8 + 25q^9 + 17q^{10} + 9q^{11} + 5q^{12}$& \\
     \hline
  $D_4$& $1 + 4q + 6q^2 + 12q^3 + 11q^4 + 13q^5 $&$12=\ell(w_\circ)$\\
       &$+ 15q^6 + 13q^7 + 2q^8 + q^9$& \\
    \hline
  $D_5$& $ 1 + 5q + 8q^2 + 16q^3 + 28q^4 + 41q^5 + 55q^6  $&$20=\ell(w_\circ)$\\
       &$+ 77q^7 + 86q^8 + 91q^9+ 109q^{10} + 90q^{11}  $& \\
       &$+ 63q^{12} + 54q^{13} + 20q^{14} + 5q^{15} $& \\
 \hline
  $F_4$& $ 1 + 4q + 6q^2 + 10q^3 + 18q^4 + 24q^5 + 38q^6 + 40q^7   $&$24=\ell(w_\circ)$\\
       &$+ 47q^8+ 54q^9 + 61q^{10} + 70q^{11} + 66q^{12} + 70q^{13}   $& \\
       &$+ 61q^{14} + 54q^{15} + 43q^{16} + 36q^{17} + 22q^{18} + 16q^{19} + 4q^{20} $& \\
 \hline
  $H_3$& $1 + 3q + 4q^2 + 6q^3 + 7q^4 + 10q^5 + 9q^6 + 10q^7  $&$15=\ell(w_\circ)$\\ 
    &$+ 8q^8 + 10q^9 + 8q^{10} + 7q^{11} + 3q^{12} + q^{13}$& \\
\hline
 \hline
  $\tilde A_2$&$1 + 3q + 6q^2 + 6q^3 + 12q^4 + 9q^5 + 18q^6   $&$12$\\ 
    &$+ 15q^7 + 24q^8 + 21q^9 + 30q^{10} + 27q^{11} + 36q^{12}$& \\
    \hline
  $\tilde B_2$& $ 1 + 3q + 4q^2 + 8q^3 + 9q^4 + 9q^5 + 14q^6    $&$12$\\ 
    &$+ 17q^7 + 19q^8 + 20q^9 + 23q^{10} + 29q^{11} + 30q^{12}$& \\
    \hline
  $\tilde G_2$& $1 + 3q + 4q^2 + 6q^3 + 7q^4 + 12q^5 + 13q^6   $&$12$\\ 
    &$+ 15q^7 + 16q^8 + 19q^9 + 23q^{10} + 23q^{11} + 27q^{12} $& \\

    \hline \hline
  \begin{tikzpicture}
	[scale=1,
	 sommet/.style={inner sep=2pt,circle,draw=black,fill=blue,thick,anchor=base},
	 rotate=0,
	 baseline = 0]
 \tikzstyle{every node}=[font=\small]
\coordinate (ancre) at (0,-0.0);
\node  at ($(ancre)+(0,0.6)$) {};
\node[sommet] (a2) at ($(ancre)+(0.25,0.5)$) {};
\node[sommet] (a3) at ($(ancre)+(0.5,0)$) {} edge[thick] node[auto,swap,right] {$ $} (a2) ;
\node[sommet] (a4) at (ancre) {} edge[thick] node[auto,swap,below] {$ $} (a3) edge[thick] node[auto,swap,left] {$4$} (a2);
\end{tikzpicture}
&$ 1 + 3q + 6q^2 + 8q^3 + 14q^4 + 20q^5 + 27q^6 + 44q^7 $&$12$\\
&$+ 56q^8+ 87q^9 + 118q^{10} + 169q^{11} + 238q^{12} $& \\

  \hline
   \begin{tikzpicture}
	[scale=1,
	 sommet/.style={inner sep=2pt,circle,draw=black,fill=blue,thick,anchor=base},
	 rotate=0,
	 baseline = 0]
 \tikzstyle{every node}=[font=\small]
\coordinate (ancre) at (0,-0.0);
\node  at ($(ancre)+(0,0.6)$) {};
\node[sommet] (a1) at ($(ancre)+(0,0.5)$) {};
\node[sommet] (a2) at ($(ancre)+(0.5,0.5)$) {} edge[thick] node[auto,swap,right] {} (a1) edge[thick] node[auto,swap,left] {} (a3);
\node[sommet] (a3) at ($(ancre)+(0.5,0)$) {} edge[thick] node[auto,swap,right] {} (a1) ;
\node[sommet] (a4) at (ancre) {} edge[thick] node[auto,swap,below] {} (a3) ;
\end{tikzpicture}
&$ 1 + 4q + 8q^2 + 14q^3 + 26q^4 + 41q^5 + 73q^6 $&$12$\\
& $+ 114q^7 + 178q^8 + 278q^9 + 422q^{10} + 631q^{11} + 939q^{12}$& \\
  \hline
\end{tabular}
  \captionof{table}{Generating functions of partition-irreducible elements up to length $k$}\label{tab:ex2}
\end{center}
\end{figure}


\bibliographystyle{plain}

\end{document}